\documentclass[reqno]{amsart}
\usepackage{amsmath,amssymb,accents,amsthm}
\usepackage[latin1]{inputenc}
\allowdisplaybreaks[4]
\newtheorem{theorem}{Theorem}[section]
\newtheorem{proposition}[theorem]{Proposition}
\newtheorem{corollary}[theorem]{Corollary}
\newtheorem{lemma}[theorem]{Lemma}

\theoremstyle{definition}
\newtheorem{example}[theorem]{Example}
\numberwithin{equation}{section}
\newcommand\bm{\begin{pmatrix}}
\renewcommand\em{\end{pmatrix}}
\newcommand\ba{\begin{array}}
\newcommand\ea{\end{array}}
\renewcommand\({\Big(}
\renewcommand\){\Big)}
\renewcommand\[{\Big[}
\renewcommand\]{\Big]}
\newcommand\er{\eqref}
\newcommand\be[1]{\begin{equation}\label{#1}}
\newcommand\ee{\end{equation}}
\newcommand\bc[1]{\begin{eqnarray*}\label{#1}}
\newcommand\ec{\end{eqnarray*}}
\newcommand\<{\langle}
\renewcommand\>{\rangle}
\newcommand\I{\int\limits}
\let\oldS\S
\renewcommand\S{\sum\limits}
\renewcommand\P{\prod\limits}
\newcommand\U{\bigcup\limits}
\newcommand\W{\bigwedge\limits}

\renewcommand\c[1]{{\check{#1}}}
\renewcommand\d[1]{{\dot{#1}}}
\newcommand\h[1]{{\hat{#1}}}
\renewcommand\r[1]{{\mathring{#1}}}

\renewcommand\t[1]{{\tilde{#1}}}

\renewcommand\o[1]{{\overline{#1}}}

\newcommand\F[1]{{\sqrt{#1}}}
\newcommand\f[2]{{\frac{#1}{#2}}}
\newcommand\la{\alpha}
\newcommand\lb{\beta}
\newcommand\lc{\chi}
\renewcommand\lg{\gamma}
\newcommand\lG{\Gamma}
\newcommand\ld{\delta}
\newcommand\lD{\Delta}
\newcommand\Le{\epsilon}
\newcommand\lh{\eta}

\renewcommand\ll{\lambda}
\newcommand\lL{\Lambda}
\newcommand\lm{\mu}
\newcommand\lM{\mho}
\renewcommand\ln{\nu}

\newcommand\lf{\phi}
\newcommand\lF{\Phi}
\renewcommand\lq{\psi}

\newcommand\lp{\pi}

\newcommand\lr{\rho}
\newcommand\ls{\sigma}

\newcommand\lT{\Theta}
\newcommand\Lt{\tau}
\newcommand\lo{\omega}
\newcommand\lO{\Omega}
\newcommand\lx{\xi}
\newcommand\lz{\zeta}
\newcommand\ui{\cap}

\newcommand\ic{\subset}

\newcommand\xx{\times}
\newcommand\xt{\otimes}
\newcommand\oo{\infty}
\newcommand\oc{\circ}

\newcommand\dl{\partial}
\newcommand\op{\oplus}
\newcommand\al{\approx}
\newcommand\el{\ell}
\newcommand\yi{\wedge}

\newcommand\DL{\mathcal D}

\newcommand\FL{\mathcal F}

\newcommand\HL{\mathcal H}
\newcommand\IL{\mathcal I}

\newcommand\KL{\mathcal K}
\newcommand\LL{\mathcal L}

\newcommand\PL{\mathcal P}

\newcommand\SL{\mathcal S}
\newcommand\TL{\mathcal T}

\newcommand\Bl{{\mathbf B}}
\newcommand\Cl{{\mathbf C}}
\newcommand\Dl{{\mathbf D}}

\newcommand\Nl{{\mathbf N}}
\newcommand\Ol{{\mathbf O}}

\newcommand\Rl{{\mathbf R}}
\newcommand\Sl{{\mathbf S}}

\newcommand\Zl{{\mathbf Z}}
\newcommand\gL{\mathfrak g}
\newcommand\hL{\mathfrak h}
\newcommand\kL{\mathfrak k}

\newcommand\tL{\mathfrak t}

\newcommand\m{{\boldsymbol m}}
\newcommand\n{{\boldsymbol n}}
\newcommand\bi{\begin{itemize}}
\newcommand\ei{\end{itemize}}
\let\oldi\i
\renewcommand\i[1]{\item{#1}}
\newcommand\tr{\operatorname{tr}}
\newcommand\hf{f}
\newcommand\hq{g}
\newcommand\Ho[1]{H_{\overline{#1}}}
\newcommand\jfj[3]{{}_1\!F_1\(\begin{matrix}#1\\#2\end{matrix}\Big|\,#3\)}
\newcommand\Hy[3]{\(\begin{matrix}#1\\#2\end{matrix}\Big|\,#3\)}

\newcommand\jj{{\langle j\rangle}}
\renewcommand\Re{\operatorname{Re}}

\begin{document}

\title{Reproducing kernel functions and asymptotic expansions on Jordan-Kepler manifolds}
\author[M.~Engli{\accent20s}, H.~Upmeier]{Miroslav Engli{\accent20s} {\rm and} Harald Upmeier}
\medskip

\address{Mathematics Institute, Silesian University in Opava, Na~Rybn{\accent19\oldi}{\accent20c}ku~1, 74601~Opava, Czech Republic {\rm and} Mathematics Institute, {\accent20Z}itn{\accent19a} 25, 11567~Prague~1, Czech Republic}
\email{englis{@}math.cas.cz}

\address{Fachbereich Mathematik, Universit\"at Marburg, D-35032 Marburg, Germany}
\email{upmeier{@}mathematik.uni-marburg.de}

\thanks{The first-named author was supported by GA \accent20CR grant no.~16-25995S}
\medskip

% 14M12 Determinantal varieties
% 17C36 Jordan algebras / Associated manifolds
% 32M15 Hermitian symmetric spaces, bounded symmetric domains, Jordan algebras
% 46E22 Hilbert spaces with reproducing kernels
% 47B35 Toeplitz operators, Hankel operators, Wiener Hopf operators
% 53D55 Deformation quantization, star-products


\begin{abstract} We study the complex geometry of generalized Kepler manifolds, defined in Jordan theoretic terms, introduce Hilbert spaces of holomorphic functions defined by radial measures, and find the complete asymptotic expansion of the corresponding reproducing kernels for K\"ahler potentials, both in the flat and bounded setting.
\end{abstract}

\maketitle

\section{Introduction}
For a K\"ahler manifold, with (integral) K\"ahler form $\lo$ and quantizing line bundle $\LL,$ it is a fundamental problem to measure the deviation of the `classical' curvature $\ln\lo$ of the $\ln$-th power $\LL^\ln$ from the `quantum' curvature $\lo_\ln$ obtained by the pull-back of the Fubini-Study form on the projective space of holomorphic sections. This relationship is usually expressed in terms of an asymptotic expansion of $\lo_\ln$ in inverse powers of the deformation parameter $\ln,$ known as the TYZ-expansion of the Kempf distortion function. Closely related is the asymptotic expansion of the reproducing kernel function of the Hilbert space of holomorphic sections, a fundamental tool in complex analysis.\\
In this paper we carry out this program for an important class of algebraic varieties, the so-called Kepler varieties defined in a Jordan theoretic setting, which generalize the well-known determinantal varieties in matrix spaces. These varieties and their regular part (Kepler manifolds) are of interest from several points of view. As algebraic varieties, we show that Kepler varieties are normal,  and classify invariant holomorphic differential forms of top-degree. Using a natural polar decomposition, these manifolds carry radial measures giving rise to reproducing kernel Hilbert spaces of holomorphic functions. The most interesting radial measures on Kepler manifolds come from K\"ahler potentials, both in the flat and bounded setting. For these, the reproducing kernel functions are related
to multi-variable hypergeometric functions of type ${}_1F_1$ and ${}_2F_1.$ Among our main results is the asymptotic expansion of these hypergeometric functions, which in the multi-variable case is quite challenging, and leads to the TYZ-expansion mentioned above.

Using chains of Peirce spaces in Jordan triples, the approach presented here can be extended to yield asymptotic expansions of TYZ-type in a quite general setting of homogeneous flag varieties.

\section{Generalized Kepler manifolds}
It is well-known that hermitian bounded symmetric domains are characterized (via their holomorphic tangent space at the origin) by the so-called {\bf hermitian Jordan triples}. Geometrically, the Jordan triple product gives the holomorphic part of the Riemann curvature tensor. We use \cite{FK} as our standard reference concerning Jordan algebras and analysis on symmetric cones. For the more general Jordan triples, see \cite{L2}. Let $Z$ be an irreducible hermitian Jordan triple of rank $r.$ Let $\{u;v;w\}=:D(u,v)w$ denote the Jordan triple product of $u,v,w\in Z.$ Consider the Bergman operators
$$B(u,v)\equiv B_{u,v}:=id-D(u,v)+Q_uQ_v$$
where $Q_uz:=\f12\{u;z;u\}.$

Choose a frame $e_1,\ldots,e_r$ of minimal orthogonal tripotents. The associated `joint Peirce decomposition' \cite[Theorem 3.13]{L2} defines numerical invariants $a,b$ such that $d:=\dim Z$ satisfies
$$\f dr=1+\f a2(r-1)+b$$
and the so-called {\bf genus} is given by
\be{1m}p:=2+a(r-1)+b.\ee
We say that $Z$ has {\bf type} $(r,a,b).$ Consider the Peirce decomposition $Z=Z_e^2\oplus Z_e^1$ for the  maximal tripotent $e:=e_1+\cdots+e_r.$ Then $d=d'+d'',$ where $d':=\dim Z_e^2,\ d'':=\dim Z_e^1.$ Then
$$p-\f dr=\f{d'}r.$$
If $b=0$ then $Z=Z_e^2$ is the complexification of a euclidean Jordan algebra with unit element $e.$ We say that $Z$ is `of tube type'. For tube type triples, let $P_s=Q_sQ_e$ denote the `quadratic representation'.

Let $\r Z$ denote the open dense subset of all elements of maximal rank $r$. If $Z$ is of tube type, with Jordan determinant $N$, then
$$\r Z=\{z\in Z:\ N(z)\ne 0\}$$
consists of all invertible elements in $Z$.

Let $\r K$ denote the connected complex Lie group of linear transformations of $Z$ generated by all Bergman operators $B(u,v)$ which are invertible. The identity component $K$ of the Jordan triple automorphism group is a compact real form of $\r K$. The Lie algebra $\kL$ of $K$ is the $\Rl$-linear span of $D(u,v)-D(v,u)$, with $u,v\in Z$. The Lie algebra $\r\kL=\kL^\Cl$ of $\r K$ is the
$\Cl$-linear span of $D(u,v)$, with $u,v\in Z.$

The {\bf rank} of $z\in Z$ is the rank of its supporting tripotent, or equivalently, the rank of its Peirce 2-space considered as a unital Jordan algebra. For $1\le\el\le r$ the set
$$\r Z_\el=\{z\in Z:\ \mbox{rank}\ z=\el\}$$
is a complex manifold called the (generalized) {\bf Kepler manifold} associated with $Z$. Its closure is the {\bf Kepler variety}
$$Z_\el:=\{z\in Z:\ \mbox{rank}\ z\le\el\}.$$
The group $\r K$ preserves the rank of elements in $Z$ and acts transitively on $\r Z_\el$. Let $S_\el\ic\r Z_\el$ denote the compact real manifold of all tripotents of rank $\el$. For maximal $\el=r$ we have $\r Z_r=\r Z$ and $S_r=:S$ is the {\bf Shilov boundary} of the unit ball of $Z.$ The {\bf symmetric cone} \cite{FK} of a euclidean Jordan algebra $X$ will be denoted by $\lO.$ For any tripotent $c\in Z$ let $X_c$ denote the self-adjoint part of the Peirce 2-space $Z_c^2$ under the involution $Q_c.$ This is a euclidean Jordan algebra with unit element $c,$ symmetric cone $\lO_c$ and Jordan determinant $N_c$ normalized by $N_c(c)=1.$ The Kepler manifold has a
{\bf polar decomposition}
\be{1a} \r Z_\el=\U_{c\in S_\el}\lO_c.\ee
For $\el=1$ we have $\lO_c=\Rl_+\ c,$ where $\Rl_+:=\{t\in\Rl:\ t>0\},$ and \er{1a} simplifies to $\r Z_1=\Rl_+S_1.$ The
{\bf quasi-determinant} $\lD(z,w)=\det B(z,w)^{1/p}$ of $Z$ has an expansion
$$\lD(z,w)=\S_{i=0}^r(-1)^i\ \lD_i(z,w)$$
into $K$-invariant sesqui-polynomials $\lD_i(z,w),$ homogeneous of bi-degree $(i,i).$

\begin{lemma}\label{1c} Restricted to $X_c,$ we have $\lD_\el(x,x)=N_c(x)^2=N_c(x^2).$ Hence $\lD_\el(z,z)>0$ for all $z\in\r Z_\el.$
\end{lemma}
\begin{proof} Let $x=\S_{i=1}^\el\ll_i c_i$ be the spectral decomposition of $x,$ with $\ll_i>0$ and $c_1+\ldots+c_\el=c.$ Then
$$\lD(x,x)=\P_{i=1}^\el(1-\ll_i^2)=\S_{j=0}^\el(-1)^j\ls_j(\ll_1^2,\ldots,\ll_\el^2).$$
Thus the component of bi-degree $(\el,\el)$ is given by $\lD_\el(x,x)=\ls_\el(\ll_1^2,\ldots,\ll_\el^2)=\P_{i=1}^\el\ll_i^2=N_c(x)^2.$
\end{proof}

\begin{example}\label{1d} For $d=n+1\ge 3$ consider $Z=\Cl^d$ with coordinates $z=(z_0,\ldots,z_n)$. Put $(z|w):=2z\cdot\o w.$  Endowed with the triple product
$$\{x;y;z\}=(x|y)z+(z|y)x-(x|\o z)\o y$$
$Z$ becomes a hermitian Jordan triple of rank 2, called a {\bf spin factor}, and denoted by $Z=IV_d$. Here $a=n-1=d-2$ and $b=0$. The minimal tripotents
$e_1:=(\f12,\f i2,0\ldots,0)$ and $e_2:=\o e_1$ form a frame. Put $e:=e_1+e_2=(1,0,\ldots,0).$ Then
$$N(z):=\f12(z|\o z)=z\cdot z=z_0^2+\ldots+z_n^2$$
is the Jordan determinant normalized by $N(e)=1.$  For $\el=1$ we obtain the Kepler manifold
$$\r Z_1=\{z\in Z:\ z\ne 0,\ N(z)=0\}=\{z\in\Cl^{n+1}\setminus\{0\}:\ z\cdot z=z_0^2+\ldots+z_n^2=0\}$$
studied in \cite{E}. Its closure is
\be{1e} Z_1=\{z\in Z:\ N(z)=0\}=\{z\in\Cl^{n+1}:\ z\cdot z=z_0^2+\ldots+z_n^2=0\}.  \ee
The compact manifold $S_1$ of all rank 1 tripotents can be identified with the {\bf cosphere bundle}
$\Sl^*(\Sl^n)=\{(x,\lx):\ \|x\|=\|\lx\|=1,\ x\cdot\lx=0\}$ over $\Sl^n$ via the map
$$\Sl^*(\Sl^n)\to S_1,\quad(x,\lx)\mapsto\f{x+i\lx}2.$$
It follows that $\r Z_1=\Rl_+S_1$ is the cotangent bundle of $\Sl^n$ without the zero-section.
\end{example}

Closely related to the Kepler manifold is the {\bf Peirce manifold} $M_\el$ of $Z$, consisting of all Peirce 2-spaces in $Z$ of rank
$\el$. This is a compact complex manifold, and there is a holomorphic submersion $\r Z_\el\to M_\el$ which maps $z\in\r Z_\el$ to its
(generalized) Peirce 2-space $Z^2_z$ \cite{Sch}. Thus there is a holomorphic fibration
$$\r Z_\el=\U_{E\in M_\el}\r E$$
and a real-analytic fibration
\be{1n}S_\el=\U_{E\in M_\el}S_E,\ee
where $S_E$ denotes the Shilov boundary relative to $E.$

Let $U:=Z_c^2,\ V:=Z_c^1,\ W:=Z_c^0$ denote the Peirce spaces for the `base point'
\be{1f}c:=e_1+\ldots+e_\el\in S_\el.\ee
Then $U=X_c\op i X_c.$ The compact manifold $S_\el$ has the tangent space
$T_c(S_\el)=i X_c\op V.$ Therefore \er{1a} implies $T_c(\r Z_\el)=U\op V=W^\perp.$ It follows that
\be{1g}d_\el:=\dim\r Z_\el=d'_\el+d''_\el,\ee
where $d'_\el:=\dim U=\el(1+\f a2(\el-1))$ and
\be{1h}d''_\el=\dim V=\el(a(r-\el)+b).\ee
For $\el=1$ we obtain $d_1=\dim_\Cl\r Z_1=1+a(r-1)+b=p-1,$ where $p$ is the genus \er{1m} of $Z.$

Denote by $K^c$ the stabilizer of $c$ in~$K.$ The following result is known for classical (non-exceptional) Jordan triples
\cite[Proposition 8.3]{KZ} (with a case-by-case proof).

\begin{theorem}\label{1i} For any hermitian Jordan triple $Z$, the Kepler variety $Z_\el$ is normal.
\end{theorem}
\begin{proof} In general, let $G$ be a connected algebraic group, acting linearly on a vector space $Z$. Let $U\ic Z$ be a linear subspace and put $P:=\{g\in G:\ gU=U\}$. Assume that $M=G/P$ is complete and that the action of $P$ on $U$ is completely reducible.
Consider the homogeneous vector bundle
$$G\xx_P U=\{[g,u]=[gp,p^{-1}u]:\ g\in G,u\in U,p\in P\}$$
over $G/P$ induced by the linear action of $P$ on $U$. Then there is a $G$-equivariant morphism $f:G\xx_P U\to Z$ given by $f[g,u]=g(u)=gp(p^{-1}u)$. Under these assumptions, \cite[Theorem 0]{K} asserts that the image $f(G\xx_P U)\ic Z$ is a closed normal variety. We will apply this theorem to the hermitian Jordan triple $Z$ and the connected algebraic group $G=\r K$ acting on
$Z.$ Put $U=Z_c^2.$ Then $P=\{h\in\r K:\ hU=U\}.$ Therefore $G/P=M_\el$ is a compact, hence complete, complex manifold. The homogeneous vector bundle $\r K\xx_P U$ coincides with the {\bf tautological bundle} $\TL_\el$ over $M_\el$, whose fibre over $E\in M_\el$ is $E$ itself. The $G$-equivariant morphism $f:\TL_\el\to Z$ is the inclusion map $E\ic Z$ on each fibre $E$. Since $Z_\el$ is the union of all Peirce 2-spaces of rank $\el$ it follows that $f(\TL_\el)=Z_\el$. The identity component $L$ of the Jordan triple automorphism group of $U$ acts irreducibly on $U$. Moreover, the natural restriction homomorphism $P\to L$ given by $p\mapsto p|_U$ is surjective, since for any $h\in L$ there exists $\t h\in K$ such that $\t h U=U$ and $h=\t h|_U$. Therefore $P$ acts also irreducibly on $U$. Thus all the assumptions of \cite[Theorem 0]{K} are satisfied, and hence $Z_\el$ is a normal Cohen-Macauley variety.

We~remark that, moreover, the map $f:\TL_\el\to Z_\el$ is birational, since it is an isomorphism from $\U_{E\in M_\el}\r E$ onto the Kepler manifold. Hence, by the same Theorem, $Z_\el$ has only rational singularities.
\end{proof}

For the spin factor, the variety $Z_1$ given by \er{1e} is one of the `standard' examples of a normal variety.

For any normal variety $X$ the following `Second Riemann Extension Theorem' \cite[p.~341, Corollary]{L1} holds: Suppose the singular set $Y\ic X$ is of codimension $\ge 2$. Then every holomorphic function $f:X\setminus Y\to\Cl$ has a (unique) holomorphic extension
$\t f:X\to\Cl$. We will apply this result to the Kepler varieties $Z_\el$.

The case where $Z$ is of tube type and $\el=r$ is somewhat exceptional and will be treated separately.

\begin{theorem}\label{1j} Exclude the case $\el=r,\ b=0$. Then any holomorphic function $f:\r Z_\el\to\Cl$ on the Kepler manifold has a unique holomorphic extension to its closure $Z_\el$.
\end{theorem}
\begin{proof} By \er{1g}, the singular set $Z_{\el-1}$ of $Z_\el$ is an analytic subvariety of codimension
$d_\el-d_{\el-1}=1+b+a(r-\el)\ge 2.$ Since $Z_\el$ is normal by Theorem~\ref{1i}, the assertion follows by the Second Riemann Extension Theorem quoted above.
\end{proof}

\section{Radial measures and polar decomposition}

A $K$-invariant smooth measure $\lr$ on $\r Z_\el,$ or a suitable $K$-invariant subset, has a {\bf polar decomposition}
\be{2a}\I_{\r Z_\el}\ d\lr(z)\ f(z):=\I_{\lO_c}d\t\lr(t)\I_K dk\ f(k\F t)\ee
for all continuous compactly supported functions $f.$ Here the {\bf radial part} $d\t\lr(t)$ is a (smooth) measure on $\lO_c$ which is invariant under the automorphism group $L_c=\operatorname{Aut}(X_c)$ of the euclidean Jordan algebra $X_c.$ For $\el=1,$ \er{2a} simplifies to
\be{2b}\I_{\r Z_1}d\lr(z)\ f(z):=\I_0^\oo d\t\lr(t)\I_{S_1}ds\ f(s\F t).\ee

The basic $K$-invariant measure on $\r Z_\el$ is the {\bf Riemann measure} $\lL_\el$ induced by the hermitian metric on~$Z$. In general, a smooth map $\lg:M\to N$ between Riemannian manifolds with Riemann measure $\lL_M$ and $\lL_N,$ respectively, satisfies
$$\lg^*\lL_N=\ld\cdot\lL_M,$$
where the density $\ld$ on $M$ is given by $\ld(z):=(\det T_z(\lg)^*T_z(\lg))^{1/2}.$ If $M,N$ are complex hermitian manifolds and
$\lg$ is holomorphic, the density becomes $\ld(z)=\det(\lg'(z)^*\lg'(z))$ for the holomorphic derivative $\lg'(z).$

\begin{lemma}\label{2c} Consider the map $\lg:\lO_c\xx K/K^c\to\r Z_\el$ defined by $\lg(t,kK^c):=k t$ for all $k\in K$ and
$t\in\lO_c.$ Then $(\det T_{t,c}(\lg)^*T_{t,c}(\lg))^{1/2}=2^{-d'_\el}\ N_c(t)^{2d''_\el/\el}\ \det F'(t),$ where
$F(t)=t^2=Q_tc=\f12\{t;c;t\}$ is the square-map on $\lO_c.$
\end{lemma}
\begin{proof} Since $K/K^c=S_\el$, with base point $c,$ we have $T_c(K/K^c)=i X_c\op V.$ For $y\in iX_c$ the vector field
$\lx_y:=\f14(D(y,c)-D(c,y))\in\kL$ satisfies $\lx_yc=\f12 D(y,c)c=y$ since $D(c,y)c=-D(y,c)c$. For $v\in V$ the vector field
$\lx_v:=D(v,c)-D(c,v)\in\kL$ satisfies $\lx_vc=D(v,c)c=v$ since $D(c,v)c=0$. Let $t\in\lO_c.$ Since $D(c,y)t=-D(y,c)t$ and $D(c,v)t=0,$ it follows that
$$T_{t,c}(\lg)(y+v,x)=x+\lx_y t+\lx_v t=x+\f12\{y;c;t\}+\{v;c;t\}=x+\f12 D(t,c)y+D(t,c)v$$
for all $x\in T_t(\lO_c)=X_c.$ Thus $T_{t,c}(\lg)=id_{X_c}\oplus\f12 D(t,c)|_{i X_c}\oplus D(t,c)|_V$ is a block-diagonal matrix with respect to the orthogonal decomposition $X_c\op i X_c\op V.$ Since the square-map has the derivative $F'(t)=D(t,c)|_{X_c},$ it follows that $\det D(t,c)|_{i X_c}=\det D(t,c)|_{X_c}=\det F'(t).$ Hence
$$\det(T_{t,c}(\lg)^*T_{t,c}(\lg))^{1/2}=2^{-d'_\el}\ \det D(t,c)|_{i X_c}\ \det D(t,c)|_{V_\Rl}=2^{-d'_\el}\ \det F'(t)\ |\det D(t,c)|_V|^2.$$
The map $u\mapsto D(u,c)|_V$ is a representation of $U$ on $V,$ with $D(c,c)|_V=id.$ Thus \er{1h} implies
\be{2d}\det D(u,c)|_V=N_c(u)^{d''_\el/\el}\ee
for all $u\in U.$ Putting $u:=t$ the assertion follows.
\end{proof}

For any $1\le\el\le r$ the Peirce manifold $M_\el$ can be regarded as the conformal compactification \cite{L2} of the Peirce 1-space
$Z_c^1,$ considered as a hermitian Jordan subtriple of $Z.$ The normalized inner product on $Z$ (with minimal tripotents of length 1) induces an inner product on $Z_c^1$ and hence an invariant volume form on $\h Z_c^1\al M_\el.$ Somewhat surprisingly, the associated volume $|M_\el|$ can be expressed in a uniform way using the invariants of $Z.$

\begin{proposition}\label{2n} Let $1\le \el\le r.$ Then the volume $|M_\el|$ of the $\el$-th Peirce manifold $M_\el$ satisfies
\be{1k}\f{|M_\el|}{\lp^{d''_\el}}=\f{\lG_\el(d'_\el/\el)\ \lG_\el(a\el/2)}{\lG_\el(d/r)\ \lG_\el(ar/2)}\ee
(this yields the value 1 in the extreme case $\el=r,\ b=0$).
\end{proposition}
\begin{proof} The volume $|\h Z|$ of the conformal compactification $\h Z$ of an irreducible hermitian Jordan triple $Z$, under the invariant metric induced by the normalized inner product on $Z=T_0(\h Z),$ satisfies
\be{1l}\f{|\h Z|}{\lp^d}=\I_Z \f{dw}{\lp^d}\lD(w,-w)^{-p}=\f{\lG_r(p-d/r)}{\lG_r(p)}=\f{\lG_r(d'/r)}{\lG_r(p)}.\ee
(see, e.g., \cite[p. 5]{EU}). This can also be verified using the case-by-case formulas of \cite{AY} where, however, a different normalization is used. We will show that applying \er{1l} to $Z_c^1$ one obtains \er{1k}. This is not obvious since the type of $Z_c^1$ is quite different from the type of $Z.$ We denote by $\lG^{(a)}$ the Gindikin Gamma-function for parameter $a.$ It suffices to compare the 'reduced' volumes $|\cdot|_{red}$ obtained by dividing by $\lp^{d''_\el}.$ For the {\bf symmetric matrices} $Z=\Cl^{r\xx r}_{sym},$ the rank of
$Z_c^1=\Cl^{\el\xx(r-\el)}$ is given by $\min(\el,r-\el).$ Evaluating \er{1k} with $a=1,\ d/r=(r+1)/2$ and $d'_\el/\el=(\el+1)/2,$ we obtain, using the duplication formula and the elementary relation
$\lG_\el^{(2)}(\el)\ \lG_{n-\el}^{(2)}(n)=\lG_n^{(2)}(n)=\lG_{n-\el}^{(2)}(n-\el)\ \lG_\el^{(2)}(n)$
\be{2o}\f{\lG_\el^{(1)}(\f{\el+1}2)\ \lG_\el^{(1)}(\f\el2)}{\lG_\el^{(1)}(\f{r+1}2)\ \lG_\el^{(1)}(\f r2)}
=\P_{i=1}^\el\f{\lG(\f{\el+1}2-\f{i-1}2)\ \lG (\f \el2-\f{i-1}2)}{\lG(\f{r+1}2-\f{i-1}2)\ \lG (\f r2-\f{i-1}2)}$$
$$=2^{\el(r-\el)}\P_{i=1}^\el\f{\lG(\el+1-i)}{\lG(r+1-i)}
=2^{\el(r-\el)}\f{\lG_\el^{(2)}(\el)}{\lG_\el^{(2)}(r)}=2^{\el(r-\el)}\f{\lG_{r-\el}^{(2)}(r-\el)}{\lG_{r-\el}^{(2)}(r)}=|\h Z_c^1|_{red}\ee
since under the induced inner product minimal tripotents in $Z_c^1$ have length $\F 2,$ so the induced inner product is twice the normalized inner product relative to $Z_c^1.$ In {\bf all other cases}, minimal tripotents in $Z_c^1$ are also minimal in $Z,$ hence the normalized inner products agree. For the {\bf full matrices} $Z=\Cl^{r\xx s}$ with $r\le s,$ the factors of the product $Z_c^1=\Cl^{\el\xx(s-\el)}\xx\Cl^{(r-\el)\xx\el}$ have rank $\min(\el,s-\el)$ and $\min(\el,r-\el),$ respectively. Evaluating \er{1k} with $a=2,\ d'_\el/\el=\el$ and $d/r=s,$ yields

$$\f{\lG_\el^{(2)}(\el)}{\lG_\el^{(2)}(s)}\ \f{\lG_\el^{(2)}(\el)}{\lG_\el^{(2)}(r)}
=\f{\lG_\el^{(2)}(\el)}{\lG_\el^{(2)}(s)}\ \f{\lG_{r-\el}^{(2)}(r-\el)}{\lG_{r-\el}^{(2)}(r)}
=\f{\lG_{s-\el}^{(2)}(s-\el)}{\lG_{s-\el}^{(2)}(s)}\ \f{\lG_{r-\el}^{(2)}(r-\el)}{\lG_{r-\el}^{(2)}(r)}=|\h Z_c^1|_{red}.$$
For the {\bf anti-symmetric matrices} $Z=\Cl^{n\xx n}_{asym},$ with $n=2r+\Le$ and $\Le=0,1,$ the rank of $Z_c^1=\Cl^{2\el\xx(n-2\el)}$ is given by $\min(2\el,n-2\el).$ Evaluating \er{1k} with $a=4,\ d/r=2r-1+2\Le,\ d'_\el/\el=2\el-1$ yields
$$\f{\lG_\el^{(4)}(2\el-1)\ \lG_\el^{(4)}(2\el)}{\lG_\el^{(4)}(2r-1+2\Le)\ \lG_\el^{(4)}(2r)}
=\P_{i=1}^\el\f{\lG(2\el-1-2(i-1))\ \lG(2\el-2(i-1))}{\lG(2r-1+2\Le-2(i-1))\ \lG(2r-2(i-1))}$$
$$=\P_{j=1}^{2\el}\f{\lG(2\el+1-j)}{\lG(2r+\Le+1-j)}=\f{\lG_{2\el}^{(2)}(2\el)}{\lG_{2\el}^{(2)}(n)}
=\f{\lG_{n-2\el}^{(2)}(n-2\el)}{\lG_{n-2\el}^{(2)}(n)}=|\h Z_c^1|_{red}.$$
For the {\bf spin factor} $Z=IV_d,$ we may assume $d\ge 5$ since $IV_3=\Cl^{2\xx 2}_{sym}$ and $IV_4=\Cl^{2\xx 2}.$ Then
$Z_c^1=IV_{d-2}.$ Evaluating \er{1k} for $a=d-2,\ \el=1,\ d'_1/1=1$ yields
$$\f{\lG _1^{(d-2)}(1)\ \lG_1^{(d-2)}(\f d2-1)}{\lG _1^{(d-2)}(\f d2)\ \lG_1^{(d-2)}(d-2)}
=\f{\lG(1)\ \lG(\f d2-1)}{\lG(\f d2)\ \lG(d-2)}=\f{\lG_2^{(d-4)}(\f d2-1)}{\lG_2^{(d-4)}(d-2)}=|\h Z_c^1|_{red}.$$
The {\bf exceptional Jordan triple} $Z=\Ol_\Cl^{1\xx 2}$ has $r=2,\ a=6,\ d/r=8.$ For $\el=1$ we have
$Z_c^1=\Cl^{5\xx 5}_{asym}.$ Evaluating \er{1k} yields
$$\f{\lG_1^{(6)}(3)\ \lG _1^{(6)}(1)}{\lG_1^{(6)}(8)\ \lG _1^{(6)}(6)}=\f{\lG(3)\ \lG (1)}{\lG(8)\ \lG (6)}=\f{\lG_2^{(4)}(3)}{\lG_2^{(4)}(8)}=|\h Z_c^1|_{red}.$$
For $\el=2$ we have $Z_e^2=Z_e^1=IV_8.$ Evaluating \er{1k} with $d'_2/2=4$ yields
$$\f{\lG_2^{(6)}(4)\ \lG _2^{(6)}(6)}{\lG_2^{(6)}(8)\ \lG _2^{(6)}(6)}=\f{\lG_2^{(6)}(4)}{\lG_2^{(6)}(8)}=|\h Z_e^1|_{red}.$$
For the {\bf exceptional Jordan algebra} $Z=\HL_3(\Ol)\xt\Cl$ of rank 3, we have $Z_c^1=\Ol_\Cl^{1\xx 2}$ both for $\el=1$ and $\el=2.$ Evaluating \er{1k} with $a=8,\ d/r=9$ and $d'_1/1=1,\ d'_2/2=5$ yields
$$\f{\lG_1^{(8)}(1)\ \lG _1^{(8)}(4)}{\lG_1^{(8)}(9)\ \lG _1^{(8)}(12)}=\f{\lG_2^{(8)}(5)\ \lG _2^{(8)}(8)}{\lG_2^{(8)}(9)\ \lG _2^{(8)}(12)}=\f{\lG(1)\ \lG (4)}{\lG(9)\ \lG (12)}=\f{\lG_2^{(6)}(4)}{\lG_2^{(6)}(12)}=|\h Z_c^1|_{red}.$$
\end{proof}

\begin{corollary} The tripotent manifold $S_\el,$ considered as a closed Riemannian submanifold of $Z,$ has the volume
\be{2m}|S_\el|=2^{d'_\el}\f{\lG_\el(a\el/2)}{\lG_\el(d/r)\ \lG_\el(ar/2)}\lp^{d_\el}.\ee
\end{corollary}
\begin{proof} Since the fibration \er{1n} is orthogonal under the Riemann measure, it follows that $|S_\el|=|M_\el|\cdot|S_c|,$ where
$S_c$ is the Shilov boundary of $Z_c^2.$ Since $|S_c|=\f{(2\lp)^{d'_\el}}{\lG_\el(d'_\el/\el)}$ by \cite[Corollary X.2.5]{FK} applied to $Z_c^2,$ the assertion follows with Proposition \ref{2n}.
\end{proof}

\begin{theorem} The Riemann measure $\lL_\el$ on $\r Z_\el$ has the radial part
\be{2e}d\t\lL_\el(t)=\f{\lp^{d_\el}\ \lG_\el(a\el/2)}{\lG_\el(d/r)\ \lG_\el(ar/2)}\ N_c(t)^{d''_\el/\el}\ dt\ee
on $\lO_c.$ In particular, $d\t\lL_1(t)=\f{\lp^{p-1}\ \lG(a/2)}{\lG(d/r)\ \lG(ar/2)}\ t^{p-2}\ dt$ for $t>0.$ For the spin factor of dimension $n+1$, we obtain $d\t\lL_1(t)=\f{2\lp^n}{(n-1)!}\ t^{n-1}\ dt.$
\end{theorem}
\begin{proof} Let $ds$ be the normalized $K$-invariant measure on $S_\el.$ Then $\lO_c\xx S_\el$ has the Riemann measure
$|S_\el|\ dt\ ds.$ Therefore Lemma \ref{2c} implies
$$\I_{\r Z_\el}d\lL_\el(z)\ f(z)=\I_{\lO_c\xx S_\el}d(\lg^*\lL_\el)(t,s)\ f(\lg(t,s))=|S_\el|\I_{\lO_c}dt\I_{S_\el}ds
\ (\det T_{t,s}(\lg)^*T_{t,s}(\lg))^{1/2}\ f(\lg(t,s))$$
$$=|S_\el|\I_{\lO_c}dt\I_K dk\ (\det T_{t,kc}(\lg)^*T_{t,kc}(\lg))^{1/2}\ f(kt)
=|S_\el|\I_{\lO_c}dt\ (\det T_{t,c}(\lg)^*T_{t,c}(\lg))^{1/2}\I_K dk\ f(kt)$$
$$=\f{|S_\el|}{2^{d'_\el}}\I_{\lO_c}dt\ \det F'(t)\ N_c(t)^{2d''_\el/\el}\I_K dk\ f(kt).$$
Replacing $t$ by $\F t$ it follows that $d\t\lL_\el(t)=\f{|S_\el|}{2^{d'_\el}}\ N_c(t)^{d''_\el/\el}\ dt.$ Now the assertion follows with \er{2m}.
\end{proof}

Replacing $f$ by $f\oc\ll^{-1/2}$ for the dilation $z\mapsto\ll^{-1/2}z$ with $\ll>0,$ the left hand side of \er{2e} scales as
$\ll^{d_\el},$ whereas the right hand side scales as $\ll^{d'_\el+\el d''_\el/\el}.$ Thus the power $N_c(x)^{d''_\el/\el}$ is confirmed by this scaling argument.

Any absolutely continuous $K$-invariant measure $\lr$ on $\r Z_\el$ has a positive $K$-invariant density
$\ld$ on $\r Z_\el$ with respect to the Riemann measure. It follows that $\lr$ has the radial part
\be{2l}d\t\lr(t)=d\t\lL_\el(t)\ \ld(\F t)=\f{\lp^{d_\el}\ \lG_\el(a\el/2)}{\lG_\el(d/r)\ \lG_\el(ar/2)}\ N_c(t)^{d''_\el/\el}\ \ld(\F t)\ dt.\ee
As an example, Lemma \ref{1c} implies
$$\I_{\r Z_\el}d\lL_\el(z)\ \lD_\el(z,z)^\lb\ f(z)=\f{\lp^{d_\el}\ \lG_\el(a\el/2)}{\lG_\el(d/r)\ \lG_\el(ar/2)}\I_{\lO_c}dt\ N_c(t)^{d''_\el/\el+\lb}\I_K dk\ f(k\F t).$$

\begin{lemma}\label{2f} The holomorphic chart $\Lt:\r U\xx V\to\r Z_\el$ defined by $\Lt(u,v):=B_{v,-c}u$, satisfies
$\Lt'_{u,v}(\lx,\lh)=B_{v,-c}(\lx+D_{u,c}\lh)$ for all $\lx\in U,\lh\in V$.
\end{lemma}
\begin{proof} For $z=(u,v)\in W^\perp$ we have $Z_{\Lt_z}^0=Z_{B(v,-c)u}^0=B_{c,-v}^{-1}Z_u^0=B_{c,-v}^{-1}W$. As a consequence of
\cite[JP13]{L2} we have $D_{v,c}D_{u,c}\lh=D_{v,Q_cu}\lh$ for all $\lh\in V$, since $Q_c\lh=0$. With $Q_cD_{u,c}\lh=0$ it follows that
$$B_{v,-c}D_{u,c}\lh=D_{u,c}\lh+D_{v,c}D_{u,c}\lh+Q_vQ_cD_{u,c}\lh=D_{u,c}\lh+D_{v,c}D_{u,c}\lh=D_{u,c}\lh+D_{v,Q_cu}\lh.$$
Therefore $\Lt'_{u,v}(\lx,\lh)=\lx+D(\lh,c)u+D(v,c)\lx+Q_vQ_c\lx+D(v,Q_cu)\lh=B_{v,-c}\lx+D(u,c)\lh+D(v,Q_cu)\lh=B_{v,-c}(\lx+D_{u,c}\lh)$ for all $\lx\in U,\lh\in V.$
\end{proof}

Since $d\ll(z)=du\ dv$ is the Lebesgue measure on $W^\perp=U\oplus V,$ the Riemann measure $\lL_\el$ on $\r Z_\el$ has the pull-back
\be{2g}\Lt^*\lL_\el=\det({\Lt'_z}^*\Lt'_z)\ d\ll(z).\ee

Interesting $K$-invariant measures arise from {\bf K\"ahler potentials}. A~strictly-plurisubharmonic function $\lf$ on a complex manifold of dimension $n$ induces a K\"ahler form $\lo=\dl\o\dl\lf,$ with associated $(n,n)$-form $\lo^n$ and measure denoted by
$|\lo^n|.$ If $\lo=\S_{i,j}\lo_{ij}(z)\ dz_i\yi d\o z_j$ is a (1,1)-form on $W^\perp$, then
$$\f{\lo^n}{n!}=\det\(\lo_{ij}(z)\)\ dz_1\yi d\o z_1\yi\ldots\yi dz_n\yi d\o z_n.$$
If $\lq$ is a smooth function on $W^\perp$, then
$\dl\o\dl\lq(z)=\S_{i,j}\dl_i\o\dl_j\lq(z)\ dz_i\yi d\o z_j$ is a (1,1)-form and
$$\f{(\dl\o\dl\lq)^n}{n!}=\det\(\dl_i\o\dl_j\lq(z)\)\ dz_1\yi d\o z_1\yi\ldots\yi dz_n\yi d\o z_n.$$

We will study two cases where symplectic potentials arise naturally. Define a {\bf bounded version} of the Kepler manifold (`Kepler ball') by taking the intersection with the bounded symmetric domain (spectral unit ball) $\c Z$ of the underlying hermitian Jordan triple $Z.$

\begin{proposition}\label{2h} (i) Consider the plurisubharmonic function $\lf_\el(w):=(w|w)^\la$ on $\r Z_\el,$ where $\la>0.$ Then
$$\f{|(\dl\o\dl\lf_\el)^{d_\el}|}{d_\el!\ 2^{d_\el}}=\la^{d_\el+1}(w|w)^{d_\el(\la-1)}\ d\lL_\el(w).$$
(ii) Consider the plurisubharmonic function $\lf_\el(w):=\log\lD(w,w)^{-p}$ on the Kepler ball $\r Z_\el\ui\c Z.$ Then
$$\f{|(\dl\o\dl\lf_\el)^{d_\el}|}{d_\el!\ 2^{d_\el}}=\lD(w,w)^{-p}\ d\lL_\el(w).$$
\end{proposition}
\begin{proof} Consider the holomorphic chart $\Lt$ defined in Lemma~\ref{2f}, and let $z_i$ be an orthonormal basis of $W^\perp$. Suppose that $\det\(\dl_i\o\dl_j(\lf\oc\Lt)(z)\)=\det({\Lt'_z}^*\Lt'_z)\ F(\Lt_z)$ for some smooth function $F.$ Then \er{2g} implies
$$\Lt^*\f{|(\dl\o\dl\lf)^{d_\el}|}{d_\el!\ 2^{d_\el}}=\f{|(\dl\o\dl(\lf\oc\Lt))^{d_\el}|}{d_\el!\ 2^{d_\el}}
=\det\(\dl_i\o\dl_j(\lf\oc\Lt)(z)\)\ d\ll(z)$$
$$=F(\Lt_z)\ \det({\Lt'_z}^*\Lt'_z)\ d\ll(z)=F(\Lt_z)\ \Lt^*\lL_\el=\Lt^*\(F(w)\ d\lL_\el(w)\).$$
(i) In the first case we have
$$\f{\dl_u\o\dl_v(\lf_\el\oc\Lt)(z)}{\la(\Lt_z|\Lt_z)^{\la-1}}=(\Lt'_zu|\Lt'_zv)+\f{\la-1}{(\Lt_z|\Lt_z)}(\Lt'_zu|\Lt_z)(\Lt_z|\Lt'_zv)
=\(u\Big|{\Lt'_z}^*\(id+\f{\la-1}{(\Lt_z|\Lt_z)}\Lt_z\Lt_z^*\)\Lt'_zv\),$$
where $\lx\lh^*v:=\lx(\lh|v).$ The inner `matrix' $id+\f{\la-1}{(\Lt_z|\Lt_z)}\Lt_z\Lt_z^*$ is the identity on $(Ran\ \Lt'_z)^\perp=Z_{\Lt_z}^0$ since $\Lt_z^*(Z_{\Lt_z}^0)=(\Lt_z|Z_{\Lt_z}^0)=\{0\}.$ It follows that
$$\f1{\la^{d_\el}\ (\Lt_z|\Lt_z)^{d_\el(\la-1)}}\f{\det\(\dl_i\o\dl_j(\lf_\el\oc\Lt)(z)\)}{\det({\Lt'_z}^*\Lt'_z)}
={\det}_{T_{\Lt_z}(\r Z_\el)}\(id+\f{\la-1}{(\Lt_z|\Lt_z)}\Lt_z\Lt_z^*\)$$
$$={\det}_Z\(id+\f{\la-1}{(\Lt_z|\Lt_z)}\Lt_z\Lt_z^*\)=1+\f{\la-1}{(\Lt_z|\Lt_z)}\tr\Lt_z\Lt_z^*=\la.$$
(ii) In the bounded case, by \cite{L2}, the Bergman metric on $\c Z$ is given by $(B_{z,z}^{-1}u|v)$ for $z\in\c Z$ and tangent vectors $u,v\in T_z(\c Z)=Z.$ Hence $\dl_u\o\dl_v\log\lD(z,z)^{-p}=\dl_u\o\dl_v\log\det B(z,z)^{-1}=(B_{z,z}^{-1}u|v).$ Since the chart $\Lt$ is holomorphic, we obtain for $\lx,\lh\in W^\perp$
$$\dl_\lx\o\dl_\lh(\lf_\el\oc\Lt)(z)=(B_{\Lt_z,\Lt_z}^{-1}\Lt'_z\lx|\Lt'_z\lh)=(\lx|{\Lt'_z}^*B_{\Lt_z,\Lt_z}^{-1}\Lt'_z\lh).$$
For any $w\in\r Z_\el$ we have $T_w(\r Z_\el)=Z^2_w\op Z^1_w,$ and
$B_{w,w}|_{Z^0_w}=id.$ It follows that
$$\f{\det\(\dl_i\o\dl_j(\lf_\el\oc\Lt)(z)\)}{\det({\Lt'_z}^*\Lt'_z)}={\det}_{T_{\Lt_z}(\r Z_\el)}B_{\Lt_z,\Lt_z}^{-1}={\det}_Z\ B_{\Lt_z,\Lt_z}^{-1}=\lD(\Lt_z,\Lt_z)^{-p}.$$
\end{proof}

Associated with the K\"ahler potentials $\lf_\el$ we consider, for $\ln>0,$ the $K$-invariant measures
\be{2k}\lr_\el^\ln:=e^{-\ln\lf_\el(w)}\f{|(\dl\o\dl\lf_\el)^{d_\el}|}{d_\el!(2\lp)^{d_\el}}.\ee

\begin{corollary} (i) For the potential $\lf_\el(w)=(w|w)^\la$ on $\r Z_\el$ the measure \er{2k} has the radial part
\be{2i}d\t\lr(t)=\f{\lG_\el(a\el/2)}{\lG_\el(d/r)\ \lG_\el(ar/2)}\la^{d_\el+1}\ N_c(t)^{d''_\el/\el}\ (t|c)^{d_\el(\la-1)}\ e^{-\ln(t|c)^\la}\ dt\ee
on $\lO_c.$ In particular, $d\t\lr_1^\ln(tc)=\la^p\ t^{p-2}\ t^{(p-1)(\la-1)}\ e^{-\ln t^\la}\ dt=\la^p\ t^{\la(p-1)-1}
\ e^{-\ln t^\la}\ dt$ for $t>0.$\\
(ii) For the potential $\lf_\el(w)=\log\lD(w,w)^{-p}$ on $\r Z_\el\ui\c Z,$ the measure \er{2k} has the radial part
\be{2j}d\t\lr(t)=\f{\lG_\el(a\el/2)}{\lG_\el(d/r)\ \lG_\el(ar/2)}\ N_c(t)^{d''_\el/\el}\ N_c(c-t)^{\ln-p}\ dt\ee
on $\lO_c\ui(c-\lO_c).$ In particular, $d\t\lr_1^\ln(tc)=t^{p-2}\ (1-t)^{\ln-p}\ dt$ for $0<t<1.$
\end{corollary}
\begin{proof} Since $(\F t|\F t)=(t|c)$ for all $t\in\lO_c$ and $\lD(\F t,\F t)=N_c(c-t)$ for all $t\in\lO_c\ui(c-\lO_c),$ the assertion follows with \er{2e} and Proposition~\ref{2h}.
\end{proof}

\section{Reproducing kernel Hilbert spaces}
Let $\Nl^r_+$ be the set of all partitions $\m$ of length $\le r$ and denote by $E^\m(z,w)=E^\m_w(z)$ the Fischer-Fock reproducing kernel for the Peter-Weyl space $\PL_\m(Z).$ Thus
\be{3a}e^{(z|w)}=\S_\m E^\m(z,w).\ee
Putting $|\m|:=m_1+\ldots+m_r,$ it follows that
\be{3b}\f{(z|w)^k}{k!}=\S_{|\m|=k}E^\m(z,w).\ee
Put $d_\m:=\dim\PL_\m(Z).$ Consider the {\bf Gindikin-Koecher Gamma function}
\be{3v}\lG_r(\m)=(2\lp)^{\f{r(r-1)}4a}\P_{j=1}^r\lG(m_j-\tfrac{j-1}2a)\ee
and define the generalized {\bf Pochhammer symbol}, for a complex number $s$ and $\m\in\Nl^r_+,$ by
\be{3w}(s)_\m =\f{\lG_r(\m+s)}{\lG_r(s)}=\P_{j=1}^r (s-\tfrac{j-1}2a)_{m_j}.\ee
Then \cite[Section IX.4]{FK} implies
\be{3c}E^\m(e,e)=\f{d_\m}{(d/r)_\m}.\ee
It follows that $E_e^\m(z)=\f{d_\m}{(d/r)_\m}\ \lF_\m(z),$ where $\lF_\m$ is the so-called {\bf spherical polynomial} of type $\m.$ This relationship will often be used later.

\begin{lemma}\label{3d} For partitions $(m)=(m,0,\ldots,0)$ of length $\el=1$, we have
\be{3e}E^{(m)}(z,w)=\f{(z|w)^m}{m!}\quad\forall\ z,w\in\r Z_1.\ee
\end{lemma}
\begin{proof} The vector space $\PL_\m(Z)$ is spanned by $K$-translates of the highest weight polynomial
$$N_\m(z)=N_1(z)^{m_1-m_2}\ N_2(z)^{m_2-m_3}\cdots N_r(z)^{m_r}.$$
Now suppose that $m_2>0.$ Then every $p\in\PL_\m(Z)$ vanishes on $\r Z_1$. Since $E^\m_w\in\PL_\m(Z)$ for all $w\in Z$ it follows that
$E^\m(z,w)=E^\m_w(z)=0$ if $z\in\r Z_1.$ This implies \er{3e} for $z,w\in\r Z_1,$ since the sum in \er{3b} has only terms with $m_2=0.$
\end{proof}

\begin{lemma}\label{3f} Let $z,w,c\in Z.$ Then
$$\I_K dk\ E^\m(z,kc)\ E^\n(kc,w)=\f{\ld_{\m,\n}}{d_\m}\ E^\m(c,c)\ E^\m(z,w).$$
In particular, choosing $c=e_1+\ldots+e_\el,$ we have
$$(E^\m_z|E^\n_w)_{S_\el}=\I_{S_\el}ds\ E^\m(z,s)\ E^\n(s,w)=\f{\ld_{\m,\n}}{d_\m}\ E^\m(c,c)\ E^\m(z,w).$$
In the special case $\el=1$, we have $\I_{S_1}ds\ (z|s)^m\ (s|w)^n=\f{\ld_{m,n}}{d_{(m)}}\ (z|w)^m\quad\forall\ z,w\in \r Z_1.$
\end{lemma}
\begin{proof} Applying Schur orthogonality \cite[Theorem 14.3.3]{D} to the compact group $K$ and its irreducible representation $\lp$ on
$\PL_\m(Z)$ it follows that
$$\I_K dk\ E^\m(z,ks)\ E^\n(ks,w)=\I_K dk\ (E^\m_z|E^\m_{ks})(E^\n_{ks}|E^\n_w)$$
$$=\I_K dk\ (E^\m_z|k^\lp E^\m_s)(k^\lp E^\n_s|E^\n_w)=\f{\ld_{\m,\n}}{d_\m}(E^\m_s|E^\m_s)(E^\m_z|E^\m_w)
=\f{\ld_{\m,\n}}{d_\m}\ E^\m(s,s)\ E^\m(z,w).$$
The special case $\el=1$ follows from Lemma~\ref{3d} since $(s|s)=1$ for all $s\in S_1.$
\end{proof}

From now until the conclusion of this section, let us temporarily exclude the `top rank on tube domain' case $\el=r,\ b=0.$
Since polynomials in $\PL_\m$ vanish on elements of rank $\el$ if $m_{\el+1}>0$, Theorem~\ref{1j} implies that any~holomorphic function $\hf$ on $\r Z_\el$ (or~on~a suitable $K$-invariant open subset containing the origin) has the Peter-Weyl expansion
\be{3h}\hf=\S_{\m\in\Nl^\el_+}\hf_\m,\ee
where $\Nl^\el_+:=\{\m\in\Nl^r_+:\ m_{\el+1}=0\}$ denotes the set of all partitions of length $\le\el.$ For a given coefficient `sequence' $\ls=(\ls_\m)_{\m\in\Nl^\el_+},$ let $\HL_\ls$ denote the Hilbert space of all holomorphic functions on $\r Z_\el$ (or a suitable $K$-invariant open subset), endowed with the inner product
\be{3i}(\hf|\hq)_\ls:=\S_{\m\in\Nl^\el_+}\ls_\m(\hf_\m|\hq_\m)_{S_\el}\ee
in terms of the Peter-Weyl decomposition \er{3h}. In case $\el=1$ for $\hf,\hq\in\HL_\ls$ we have
$$(\hf|\hq)_\ls=\S_{m=0}^\oo \ls_m(\hf_m|\hq_m)_{S_1}.$$

\begin{proposition}\label{3j} The reproducing kernel of $\HL_\ls$ has the {\bf Peter-Weyl decomposition}
\be{3k}\KL^\ls(z,w)=\S_{\m\in\Nl^\el_+}\f{(d'_\el/\el)_\m}{\ls_\m}\f{d_\m}{d_\m^c}\ E^\m(z,w)\quad\forall\ z,w\in\r Z_\el,\ee
where $d_\m^c:=\dim\PL_\m(Z^2_c).$ For $\el=1$ this simplifies to
\be{3l}\KL^\ls(z,w)=\S_{m\ge 0}\f{d_{(m)}}{\ls_m}\ (z|w)^m.\ee
\end{proposition}
\begin{proof} There exist constants $c_\m>0$ such that
$$\KL^\ls(z,w)=\KL_w^\ls(z)=\S_{\m\in\Nl^\el_+}c_\m\ E^\m(z,w)=\S_{\m\in\Nl^\el_+}c_\m\ E_w^\m(z).$$
The reproducing property and \er{3i} yield
$$\S_\m c_\m\ E^\m(z,w)=\KL^\ls(z,w)=(\KL_z^\ls|\KL_w^\ls)_\ls=\S_\m \ls_\m\ c_\m^2(E^\m_z|E^\m_w)_{S_\el}
=\S_\m\ \f{\ls_\m c_\m^2}{d_\m}E^\m(c,c)E^\m(z,w).$$
Comparing coefficients it follows that
$$\f1{c_\m}=\f{\ls_\m}{d_\m}\ E^\m(c,c)=\f{\ls_\m}{d_\m}\ \f{d_\m^c}{(d'_\el/\el)_\m}.$$
Here we applied \er{3c} to $Z_c^2$ and used the fact that $E^\m$ restricted to $Z_c^2$ is the Fischer-Fock kernel relative to $Z_c^2.$
For $\el=1$ use $(d'_1/1)_m=m!$ and Lemma \ref{3d}.
\end{proof}

Let $H^2(\r Z_\el,d\lr)$ denote the Hilbert space of all holomorphic functions on $\r Z_\el,$ or a suitable $K$-invariant open subset, which are square-integrable with respect to a given $K$-invariant measure $\lr.$

\begin{proposition}\label{3m} For a $K$-invariant smooth measure $\lr$ the coefficients for the Hilbert space
$H^2(\r Z_\el,d\lr)$ are given by the {\bf moments}
$$\ls_\m=\I_{\lO_c}d\t\lr\ N_\m.$$
For $\el=1$ we obtain $\ls_m=\I_0^\oo d\t\lr(t)\ t^m.$
\end{proposition}
\begin{proof} Lemma~\ref{3f} implies $(E^\m_z|E^\m_w)_{S_\el}=\f1{d_\m}\ E^\m(c,c)\ E^\m(z,w).$ Since
$$(E^\m_z|E^\m_w)_\lr=\I d\lr(\lz)\ \o{E^\m_z(\lz)}\ E^\m_w(\lz)=\I d\t\lr(t)\I_K dk\ \o{E^\m_z(k\F t)}\ E^\m_w(k\F t)
=\f{E^\m(z,w)}{d_\m}\I d\t\lr(t)\ E_e^\m(t)$$
the associated coefficients \er{3i} are given by
$$\ls_\m=\f1{E^\m(c,c)}\I_{\lO_c}d\t\lr(t)\ E^\m(t,e)=\I_{\lO_c}d\t\lr(t)\ N_\m(t).$$
The last equality is proved as follows: For $\m\in\Nl_+^\el$ the restriction of $N_\m$ to $X_c$ coincides with the conical function
$N_\m^c$ relative to $X_c$. The two integrals agree since
\be{3n}\I_{L_c}dk\ N_\m(kt)=\f{E^\m(t,c)}{E^\m(c,c)}=\f{E^\m(t,e)}{E^\m(c,c)}\ee
for all $t\in X_c,$ and $d\t\lr$ is invariant under the Jordan algebra automorphism group $L_c$ of $X_c.$
\end{proof}

Combining these results with \er{3k} or \er{3l}, respectively, yields the `Peter-Weyl decomposition' of the reproducing kernel of the Hilbert space $H^2(\r Z_\el,d\lr)$ associated with a $K$-invariant measure $\lr.$

We now consider $K$-invariant measures arising from a K\"ahler potential $\lf(w).$

\begin{lemma} \label{3o} Let $X$ be a euclidean Jordan algebra of dimension $d$ and rank $r.$ Let $\lf\in\PL(X)$ be an $m$-homogeneous polynomial. Then we have, for $Re(\lz)\ge 0,$
$$\I_{\lO}dx\ e^{-(x|e)}(x|e)^\lz\ \lf(x)=\f{\lG(m+d+\lz)}{\lG(m+d)}\I_{\lO}dx\ e^{-(x|e)}\ \lf(x).$$
More generally, for $\la>0,$
$$\la\I_{\lO}dx\ e^{-(x|e)^\la}\ (x|e)^\lz\ \lf(x)=\f{\lG\(\f{m+d+\lz}\la\)}{\lG(m+d)}\I_{\lO}dx\ e^{-(x|e)}\ \lf(x).$$
In particular, for each partition $\m\in\Nl_+^r,$ we have
\be{3t}\f{\lG(|\m|+d)}{\lG\(\f{|\m|+d+\lz}\la\)}\la\I_{\lO}dx\ e^{-(x|e)^\la}\ (x|e)^\lz\ N_\m(x)
=\I_{\lO}dx\ e^{-(x|e)}\ N_\m(x)=\lG_r\(\m+\f dr\)=\lG_r(d/r)\ (d/r)_\m.\ee
\end{lemma}
\begin{proof} Write $X=\Rl e\op Y,$ with $Y=e^\perp.$ Then $\lf(te+y)=\S_{k=0}^m t^k\ \lf_k(y),$ where $\lf_k\in\PL(Y)$ are $(m-k)$-homogeneous. Note that
$$Y_t:=\{y\in Y:\ te+y\in\lO\}=Y\ui(\lO-te)=t(Y\ui(\lO-e))=tY_1.$$
Thus for $t>0$
$$\IL_k(t):=\I_{Y_t}dy\ \lf_k(y)=\I_{Y_1}dy_1\ t^{d-1}\ t^{m-k}\lf_k(y_1)=t^{d-1+m-k}\ \IL_k(1).$$
Therefore
$$F_k(\lz):=\I_0^\oo dt\ e^{-rt}\ (rt)^\lz\ t^k\I_{Y_t}dy\ \lf_k(y)=\I_0^\oo dt\ e^{-rt}\ (rt)^\lz\ t^k\ \IL_k(t)$$
$$=\I_0^\oo\f{dt}t\ e^{-rt}\ (rt)^\lz\ t^{d+m}\ \IL_k(1)=\f{\lG(m+d+\lz)}{r^{m+d}}\ \IL_k(1).$$
It follows that
$$\I_{\lO}dx\ e^{-(x|e)}(x|e)^\lz\ \lf(x)=\I_0^\oo dt\ e^{-rt}\ (rt)^\lz\I_{Y_t}dy\ \lf(te+y)=\S_{k=0}^m F_k(\lz)
=\f{\lG(m+d+\lz)}{r^{m+d}}\ \S_{k=0}^m\IL_k(1).$$
Putting $\lz=0$ in this relation and taking quotients, the first claim follows. For parameter $\la>0$ we put $rs:=(rt)^\la.$ Then
$(t/s)^\la=(rs)^{1-\la}$ and $\la dt=(t/s)ds.$ Write $te+y=\f ts(se+v).$ Then $y:=(t/s)v$ satisfies $dy=(t/s)^{d-1}\ dv$ and $y\in Y_t$ if and only if $v\in Y_s.$ Since $\lf(te+y)=(t/s)^m\ \lf(se+v)$ and $(te+y|e)=rt$ it follows that
$$\la\I_{\lO}dx\ e^{-(x|e)^\la}\ (x|e)^\lz\ \lf(x)=\la\I_0^\oo dt\ e^{-(rt)^\la}\ (rt)^\lz\I_{Y_t}dy\ \lf(te+y)$$
$$=\I_0^\oo ds\ e^{-rs}\ (rs)^{(\lz+(1-\la)(d+m))/\la}\I_{Y_s}dv\ \lf(se+v)
=\I_{\lO}dx\ e^{-(x|e)}\ (x|e)^{(\lz+(1-\la)(d+m))/\la}\lf(x).$$
Since $\f{\lz+(1-\la)(d+m)}\la+d+m=\f{\lz+d+m}\la,$ the second assertion is reduced to the first case.
The last assertion follows with \cite[VII.1.1]{FK}.
\end{proof}

\begin{theorem} (i) For the potential $\lf_\el(w)=(w|w)^\la,\ \la>0$ the Hilbert space $H^2(\r Z_\el,\lr)$ for the measure \er{2k} has the moments
\be{3p}\ls_\m=\la^{d_\el}\ \f{\lG_\el(a\el/2)}{\lG_\el(d/r)\ \lG_\el(ar/2)}\ \f{\lG_\el(\m+d_\el/\el)}{\ln^{d_\el+|\m|/\la}}
\ \f{\lG(d_\el+\f{|\m|}\la)}{\lG(d_\el+|\m|)}.\ee
In the flat case $\la=1$ this simplifies to
$$\ls_\m=\f{\lG_\el(a\el/2)}{\lG_\el(d/r)\ \lG_\el(ar/2)}\ \f{\lG_\el(\m+d_\el/\el)}{\ln^{d_\el+|\m|}}.$$
For $\el=1$ we have
$$\ls_m=\f{\la^{p-1}\ \lG(p-1)}{\ln^{p-1}}\ \f{\lG(p-1+\f m\la)}{\lG(p-1+m)}\ \f{(p-1)_m}{\ln^{m/\la}}
=\f{\la^{p-1}}{\ln^{p-1}}\ \f{\lG(p-1+\f m\la)}{\ln^{m/\la}}.$$
(ii) For the potential $\lf_\el(w):=\log\lD(w,w)^{-p},$ the Hilbert space $H^2(\r Z_\el\ui\c Z,d\lr)$ for the measure \er{2k} has the moments
\be{3q}\ls_\m=\f{\lG_\el(a\el/2)}{\lG_\el(d/r)\ \lG_\el(ar/2)}\ \f{\lG_\el(\ln-d_\el/\el)\ \lG_\el(\m+d_\el/\el)}{\lG_\el(\m+\ln)}.\ee
\end{theorem}
\begin{proof} (i) Put $y:=\ln^{1/\la}t.$ In view of \er{2i} it follows that $\f{\lG_\el(d/r)\ \lG _\el(ar/2)}{\lG_\el(a\el/2)}\ls_\m$ equals
\be{3r}\la^{d_\el+1}\I_{\lO_c}dt\ e^{-\ln(t|c)^\la}\ (t|c)^{d_\el(\la-1)}\ N_{\m+d''_\el/\el}(t)
=\f{\la^{d_\el+1}}{\ln^{d_\el+\f{|\m|}\la}}\I_{\lO_c}dy\ e^{-(y|c)^\la}\ (y|c)^{d_\el(\la-1)}\ N_{\m+d''_\el/\el}(y).\ee
Since $|\m+\f{d''_\el}\el|+d'_\el=|\m|+d_\el$ and $\m+\f{d''_\el}\el+\f{d'_\el}\el=\m+\f{d_\el}\el,$ the assertion \er{3p} follows from \er{3t} applied to $X_c.$\\
(ii) In view of \er{2j} it follows that $\f{\lG_\el(d/r)\ \lG _\el(ar/2)}{\lG_\el(a\el/2)}\ls_\m$ equals
\be{3s}\I_{\lO_c\ui(c-\lO_c)}dt\ N_c(t)^{d''_\el/\el}\ N_c(c-t)^{\ln-p}\ N_\m(t)
=\I_{\lO_c\ui(c-\lO_c)}dt\ N_c(c-t)^{\ln-p}\ N_{\m+d''_\el/\el}(t).\ee
Evaluating this beta-integral via \cite[Theorem VII.1.7]{FK}, applied to $\lO_c,$ the assertion \er{3q} follows with
$p=\f{d'_\el+d_\el}\el.$
\end{proof}

Combining these results with Propositions~\ref{3j} and~\ref{3m} yields the Peter-Weyl decomposition of the reproducing kernels~$\KL^\ln.$

We now treat briefly the {\bf exceptional case} where $Z$ is of tube type and $\el=r.$ The corresponding determinant function $N$ does not vanish on $\r Z_r$. Thus all negative powers of $N$ are holomorphic on $\r Z_r$ but do not have a continuous extension to $Z=Z_r$. In this case the codimension is $d_r-d_{r-1}=1$ and the Peirce manifold $M_r=\{Z\}$ is just a point. The expansion \er{3h} gets replaced~by
$$\hf=\S_{m\in\Zl^r_+}\hf_\m,$$
where the summation extends over all of $\Zl^r_+:=\{(m_1,m_2,\dots,m_r)\in\Zl^r:\ m_1\ge m_2\ge\dots\ge m_r\}$ (i.e.~`negative' signatures are also allowed); see~\cite[Chapter~XII.3]{FK}. Here $\hf_\m$ belongs to the space $\PL_\m(Z)$ defined for $m_r<0$~as
$$\PL_\m(Z)=N^{m_r}\PL_{\m-m_r}(Z),$$
where $N$ is the determinant function on~$Z$. Similarly, one extends the definition of the spherical polynomials $\lF_\m$ on $Z$ to
$m_r<0$ by setting
$$\lF_\m(z)=N(z)^{m_r}\lF_{\m-m_r}(z)$$
for $z\in\r Z_r=\{z\in Z:\ N(z)\ne 0\}$. The~functions
$$H^\m(z,w):=(d/r)_\m E^\m(z,w),\qquad\m\in\Nl^r_+,$$
satisfy
\be{3u}H^{\m+k}(z,w)=N(z)^k H^\m(z,w)\o{N(w)}^k,\qquad\forall\m\in\Nl^r_+, \forall k\in\Nl.\ee
Setting
$$H^\m(z,w):=N(z)^{m_r}H^{\m-m_r}(z,w)\o{N(w)}^{m_r},\qquad\m\in\Zl^r_+,$$
thus extends the definition of $H^\m$ also to $\m\in\Zl^r_+$, and \er{3u} still holds (with any $\m\in\Zl^r_+,\ k\in\Zl$).
Note that $H^\m(z,e)=d_\m\lF_\m(z)$, where as before $d_\m=\dim\PL_\m(Z)$.

With this notation, all the results above in this section carry over also to the present setting. The~proofs are straightforward modifications of the ones given previously, hence omitted.

In~particular, for a coefficient `sequence' $\ls=(\ls_\m)_{\m\in\Zl^r_+}$ of positive numbers, consider the Hilbert space $\HL_\ls$ of holomorphic functions on $\r Z_r$  $($or a suitable $K$-invariant open subset$)$, endowed with the inner product
$$(\hf|\hq)_\ls:=\S_{\m\in\Zl^r_+}\ls_\m(\hf_\m|\hq_\m)_{S_r},$$
where $\hf=\S_{\m\in\Zl^r_+}\hf_\m$ is the Peter-Weyl decomposition of a holomorphic function $\hf$.

\begin{proposition} $\HL_\ls$ has the reproducing kernel
$$\KL^\ls(z,w)=\S_{\m\in\Zl^r_+}\f{H^\m(z,w)}{\ls_\m},\qquad z,w\in\r Z_r.$$
As a special case, for the Hilbert space $H^2(\r Z_r,d\lr)$ of all holomorphic functions on $\r Z_r$ (or~a suitable $K$-invariant open subset) which are square-integrable with respect to the measure~\er{2a}, the coefficients are $\ls_\m=\I_\lO d\t\lr\ N_\m,$ and the sum extends over all $\m\in\Zl^r_+$ for which the integral is finite.
\end{proposition}
Here $\ls_\m$ can be allowed to take the value $+\oo$, in~which case $\f1{\ls_\m}$ is interpreted as~0, and $\hf_\m=0$ for any
$\hf\in\HL_\ls$. Note in particular that for $\el=r$ and $b=0$, the integrals \er{3r} and \er{3s} are infinite if some $m_j<0$ \cite[Theorem~VII.1.1]{FK}.

\section{Universal differential operator and hypergeometric functions}
In this section we express the reproducing kernels in a more conceptual way and relate them to hypergeometric functions. By polarization and $K$-invariance, it suffices to express the reproducing kernel $\KL^\ls$ on the diagonal of $\lO_c.$ Consider the
{\bf universal differential operator}
\be{4a}\DL_\el:=N_c^{\f a2(\el-r)}\dl_{N_c}^b N_c^{\f a2(r-\el-1)+b+1}\(\dl_{N_c}N_c^{\f a2}\dl_{N_c}^{a-1}\)^{r-\el}N_c^{\f a2(r-\el+1)-1}\ee
on $\lO_c,$ which has order $\el((r-\el)a+b)=d''_\el$ and is independent of the Hilbert space chosen.
% By the Faraut-Koranyi formula and \er{EMC}, the `partial' Cauchy-Szeg\"o kernel $\lD(z,w)^{-d'_\el/\el},$ restricted to $t$ in the unit ball of~$\lO_c,$ has an expansion
% \be{SM}\lD(\F t,\F t)^{-d'_\el/\el}=\S_{\m\in\Nl_+^\el}(d'_\el/\el)_\m\ E_\m(t,c).\ee
% For $\el=1$ we have $\lD(z,w)=(1-z\o w)^{-1}$ on $Z_c^2,$ and \er{SM} simplifies to
% $$(1-t)^{-1}=\S_{m=0}^\oo t^m$$
% for $0<t<1.$

The following theorem is our main theorem in the general setting.
\begin{theorem}\label{4f} Consider coefficients $(\ls_\m)_{\m\in\Nl_+^\el},$ with $\ls_\m=\I_{\lO_c}d\t\lr\ N_\m$ for radial measures. Then the reproducing kernel $\KL$ satisfies
\be{4h}\KL^\ls(\F t,\F t)=\f{\lG_\el(a\el/2)}{\lG_\el(d/r)\ \lG_\el(ar/2)}\ \DL_\el\(\S_{\m\in\Nl_+^\el}\f{\lG_\el(\m+d'_\el/\el)}{\ls_\m}\ E_c^\m(t)\).\ee
\end{theorem}
\begin{proof} By \cite[Lemma 2.6 and Lemma 2.7]{U2} we have $d_\m:=\dim \PL_\m(Z)=\f{(d/r)_\m}{(d'/r)_\m}\ d'_\m,$ where
\be{4c}d'_\m=\P_{i<j}\f{m_i-m_j+\f a2(j-i)}{\f a2(j-i)}\f{(m_i-m_j+1+\f a2(j-i-1))_{a-1}}{(1+\f a2(j-i-1))_{a-1}}\ee
is the `tube type' dimension and the factor
$$\f{(d/r)_\m}{(d'/r)_\m}=\f{(1+\f a2(r-1)+b)_\m}{(1+\f a2(r-1))_\m}=\P_{j=1}^r \f{(1+\f a2(r-j)+m_j)_b}{(1+\f a2(r-j))_b}$$
occurs only for non-tube type domains ($b>0$). If $m_j=0$ for $j>\el$ we obtain
$$\f{d_\m}{d_\m^c}=\P_{i=1}^\el \f{(m_i+1+\f a2(r-i))_b}{(1+\f a2(r-i))_b}
\P_{1\le i\le\el<j\le r}\f{m_i+\f a2(j-i)}{\f a2(j-i)}\f{(m_i+1+\f a2(j-i-1))_{a-1}}{(1+\f a2(j-i-1))_{a-1}}=\f{A_\m}{A_0},$$
where $A_\m:=\P_{i=1}^\el(m_i+1+\f a2(r-i))_b\P_{1\le i\le\el<j\le r}(m_i+\f a2(j-i))(m_i+1+\f a2(j-i-1))_{a-1}.$ We have
\be{4g}\P_{i=1}^\el(m_i+1+\f a2(r-i))_b=\P_{i=1}^\el\f{\lG(m_i+1+b+\f a2(r-i))}{\lG(m_i+1+\f a2(r-i))}
=\f{\lG_\el(\m+d/r)}{\lG_\el(\m+d'/r)},\ee
since $\f dr=1+b+\f a2(r-1)$ and $\f{d'}r=1+\f a2(r-1).$ Moreover, for each $i\le\el,$
$$\P_{j=\el+1}^r(m_i+\f a2(j-1))\ (m_i+1+\f a2(j-i-1))_{a-1}=\P_{j=\el+1}^r(m_i+\f a2(j-1))\f{\lG(m_i+\f a2(j-i+1))}
{\lG(m_i+1+\f a2(j-i-1))}$$
$$=\P_{j=\el+1}^r\f{m_i+\f a2(j-1)}{m_i+\f a2(j-i-1)}\ \f{\lG(m_i+\f a2(j-i+1))}{\lG(m_i+\f a2(j-i-1))}
=\f{m_i+\f a2(r-1)}{m_i+\f a2(\el-1)}\ \f{\lG(m_i+\f a2(r-i))}{\lG(m_i+\f a2(\el-i))}\ \f{\lG(m_i+\f a2(r+1-i))}{\lG(m_i+\f a2(\el+1-i))}$$
$$=\f{\lG(1+m_i+\f a2(r-i))}{\lG(1+m_i+\f a2(\el-i))}\ \f{\lG(m_i+\f a2(r+1-i))}{\lG(m_i+\f a2(\el+1-i))}.$$
Taking the product over $1\le i\le\el$ yields, together with \er{4g},
\be{4b}A_\m=\f{\lG_\el(\m+d/r)}{\lG_\el(\m+d'/r)}\ \f{\lG_\el(\m+d'/r)\ \lG_\el(\m+ar/2)}{\lG_\el(\m+d'_\el/\el)\ \lG_\el(\m+a\el/2)}
=\f{\lG_\el(\m+d/r)\ \lG_\el(\m+ar/2)}{\lG_\el(\m+d'_\el/\el)\ \lG_\el(\m+a\el/2)}.\ee
Putting $\<\lb\>^\m:=\P_{i=1}^\el(\lb+m_i-\f a2(i-1)),$ we also have
$$A_\m=\P_{t=1}^b\<t+\f a2(r-1)\>^\m\P_{j>\el}\<\f a2(j-1)\>^\m\P_{s=1}^{a-1}\<s+\f a2(j-2)\>^\m.$$
In view of \er{3n} the $L_c$-invariant differential operator $D_\la:=N_c^{1-\la}\dl_{N_c}N_c^\la$ of order $\el$ on $\lO_c$ satisfies
\be{4u}D_\la\ E_c^\m=\<\la+\f a2(\el-1)\>^\m\ E_c^\m\ee
for all $\m\in\Nl_+^\el,$ since $D_\la\ N_\m=\<\la+\f a2(\el-1)\>^\m\ N_\m$ by \cite[p. 296]{FK1} applied to $X_c.$ The identity
$\P_{s=1}^{a-1}D_{s+\lb}=N_c^{-\lb}\dl_{N_c}^{a-1}\ N_c^{a+\lb-1}$ yields a factorization
\be{4q}\DL_\el=N_c^{\f a2(\el-r)}\ \dl_{N_c}^b\ N_c^{\f a2(r-\el)+b}
\ \P_{k=1}^{r-\el}N_c^{1-\f a2 k}\ \dl_{N_c}\ N_c^{\f a2}\ \dl_{N_c}^{a-1}\ N_c^{\f a2(1+k)-1}$$
$$=N_c^{\f a2(\el-r)}\dl_{N_c}^b N_c^{\f a2(r-\el)+b}
\P_{k=1}^{r-\el}D_{\f a2 k}\ N_c^{\f a2(1-k)}\ \dl_{N_c}^{a-1}\ N_c^{\f a2(1+k)-1}
=\P_{t=1}^b D_{t+\f a2(r-\el)}\P_{k=1}^{r-\el}D_{\f a2 k}\(\P_{s=1}^{a-1} D_{s+\f a2(k-1)}\)$$
$$=\(\P_{t=1}^b D_{t+\f a2(r-\el)}\)\P_{j=\el+1}^r\(D_{\f a2(j-\el)}\P_{s=1}^{a-1} D_{s+\f a2(j-\el-1)}\).\ee
Therefore $\DL_\el\ E_c^\m=A_\m\ E_c^\m=A_0\ \f{d_\m}{d_\m^c}\ E_c^\m.$ Expressing $A_0$ via \er{4b}, the assertion follows with
\er{3k}, since $E^\m(\F t,\F t)=E^\m_c(t).$
\end{proof}

% \begin{remark} (preliminary) Without passing to the diagonal, Theorem \ref{4f} can be rephrased as
% $$\KL^\ls(z,w)=(\EL\lF^\ls)(\lD_1(z,w),\ldots,\lD_\el(z,w))$$
% where $\lF^\ls(t_1,\ldots,t_\el)$ is a function in $\el$ variables and $\EL$ is a suitable differential operator in
% $\f{\dl}{\dl t_1},\ldots,\f{\dl}{\dl t_\el}.$
% Indeed, consider the function of $z,w\in Z_\el$
% $$F^\ls(z,w)=\S_{\m\in\Nl^\al_+} \f{(d'_\el/\el)_\m}{\ls_\m}E^\m(z,w).$$
% For $z=w=k(t_1e_1+\dots+t_\el e_\el)$ with $k\in K$ and $t_1,\dots,t_\el\in\Rl$, $a_j(z,w)$ are (up~to a~sign) the elementary symmetric polynomials in $t_1^2,\dots,t_\el^2$, while $F^\ls(z,w)$ is a symmetric function of $t_1^2,\dots,t_\el^2$. Since a symmetric function can be expressed as a function in elementary symmetric polynomials, we~have --- abusing notation a little ---
% \be{TEMP}F^\ls(z,w)=F^\ls(a_1(z,w),\dots,a_\el(z,w))\quad\text{for }z=w=k(t_1e_1+\dots+t_\el e_\el)\ee
% for some function $F^\ls$ (on~the right-hand side) of~$\el$ variables. Now~both sides of the last equality are holomorphic functions on $z,\o w$; since such a function is uniquely determined by its restriction to the diagonal $z=w$, it~follows that \er{TEMP} in fact prevails for all $z,w\in Z_\el$.
% ??? How to handle differentiations ???
% \end{remark}

For $\el=1$ a slightly different representation, with a simpler proof, can be given.

\begin{theorem}\label{4i} Let $\el=1.$ Consider coefficients $(\ls_m)_{m\ge 0},$ with $\ls_m=\I_0^\oo d\t\lr(t)\ t^m$ for radial measures. Then $\HL_\ls$ has the reproducing kernel
\be{4j}\KL_1^\ls(z,w)=\f{\lG(a/2)}{\lG(d/r)\ \lG(ar/2)}\ t^{\f{1-r}2a}\ \dl_t^b\ t^{\f{r-1}2a+b}\ \P_{j=2}^r\f{D_j t+tD_j}2
\(\S_{m\ge 0}\f{t^m}{\ls_m}\)\big|_{t=(z|w)},\ee
where $D_j:=t^{a-\f a2 j}\ \dl_t^a\ t^{\f a2 j-1}$ is a differential operator of order $a.$
\end{theorem}
\begin{proof} For partitions of length 1, putting $m_1=m$ and $m_i=0$ for $i>1,$ by \er{4c} we obtain the (tube type) dimension
$$d'_{(m)}=\P_{j=2}^r\f{m+\f a2(j-1)}{\f a2(j-1)}\ \f{(m+1+\f a2(j-2))_{a-1}}{(1+\f a2(j-2))_{a-1}}
=\P_{j=2}^r\f{m+\f a2(j-1)}{\f a2(j-1)}\ \f{\lG(m+\f a2j)}{\lG(\f a2j)}\ \f{\lG(1+\f a2(j-2))}{\lG(m+1+\f a2(j-2))}$$
$$=\P_{j=2}^r\f{\lG(1+\f a2(j-2))}{a(j-1)\ \lG (\f a2j)}\P_{j=2}^r\f{(2m+a(j-1))\ \lG(m+\f a2j)}{\lG(m+1+\f a2(j-2))}
=\f1{C'}\P_{j=2}^r\(\f{\lG(m+\f a2j)}{\lG(m+\f a2(j-2))}+\f{\lG(m+1+\f a2j)}{\lG(m+1+\f a2(j-2))}\)$$
where
\be{4e}C'=a^{r-1}(r-1)!\P_{j=2}^r\f{\lG(\f a2j)}{\lG(1+\f a2(j-2))}
=2^{r-2}a\f{(r-1)!}{(r-2)!}\f{\lG(\f a2(r-1))\ \lG(\f{ar}2)}{\lG(\f a2)}=2^{r-1}\f{\lG(\f{d'}r)\ \lG(\f{ar}2)}{\lG(\f a2)}.\ee
For $\lb\in\Rl$ we have
\be{4k}\dl_t^a t^\lb=\lb(\lb-1)\cdots(\lb+1-a)\ t^{\lb-a}=\f{\lG(1+\lb)}{\lG(1+\lb-a)}\ t^{\lb-a}.\ee
It follows that
$$(D_j t+tD_j)t^m=t^{a-\f a2 j}\ \dl_t^a\ t^{m+\f a2 j}+t^{1+a-\f a2 j}\ \dl_t^a\ t^{m-1+\f a2 j}
=\(\f{\lG(m+\f a2j)}{\lG(m+\f a2(j-2))}+\f{\lG(m+1+\f a2j)}{\lG(m+1+\f a2(j-2))}\)\ t^m.$$
Since $d_{(m)}=\f{(d/r)_m}{(d'/r)_m}\ d'_{(m)}=\f1C\ (\f{r-1}2a+1+m)_b\ d'_{(m)},$ with $C:=C'\ (1+\f a2(r-1))_b=2^{r-1}\f{\lG(d/r)\ \lG(ar/2)}{\lG(a/2)},$ it follows that
$$\f1C\(t^{\f{1-r}2a}\ \dl_t^b\ t^{\f{r-1}2a+b}\)\P_{j=2}^r(D_j t+tD_j)t^m=d_{(m)} t^m.$$
In view of \er{3l}, the assertion follows since $(\F t c|\F t c)=t.$
\end{proof}
Since $t^{\f{1-r}2a}\ \dl_t^b\ t^{\f{r-1}2a+b}\P_{j=2}^r(D_j t+tD_j)=2^{r-1}\DL_1,$ \er{4j} agrees with Theorem~\ref{4f}.

Let $X$ be a euclidean Jordan algebra, of dimension $d$ and rank~$r.$ For complex parameters $\la_1,\ldots,\la_p$ and
$\lb_1,\ldots,\lb_q,$ the {\bf hypergeometric function} is the formal power series
\be{4l}{}_p F_q\Hy{\la_1,\ldots,\la_p}{\lb_1,\ldots,\lb_q}z=\S_{\m\in\Nl^r_+}\f{(\la_1)_\m\ldots(\la_p)_\m}{(\lb_1)_\m\ldots(\lb_1)_\m}\ E^\m(z,e)=\S_{\m\in\Nl^r_+}\f{(\la_1)_\m\ldots(\la_p)_\m}{(\lb_1)_\m\ldots(\lb_1)_\m}\ \f{d_\m}{(\f dr)_\m}\lF_\m(z),\ee
assuming that the Pochhammer symbols in the denominator don't vanish for all $\m\in\Nl_+^r.$ For $p\le q$ this series defines an entire function on $X^\Cl;$ for $p=q+1$ it converges on the unit ball $\c X{}^\Cl,$ and for $p>q$ it has empty domain of convergence. For
$p=q=1$ we obtain the {\bf confluent hypergeometric function}
$${}_1F_1\Hy\la\lb z=\S_{\m\in\Nl^r_+}\f{(\la)_\m}{(\lb)_\m}\ E^\m(z,e)\quad\forall\ z\in X^\Cl.$$

\begin{corollary}\label{4m} For the K\"ahler potentials $\lf_\el$ introduced above the reproducing kernel $\KL_\el^\ln$ associated with the measure \er{2k} satisfies
$$\KL_\el^\ln(\F t,\F t)=\DL_\el\ \FL_\el^\ln(t),$$
where (i) for the potential $\lf_\el(w)=(w|w)^\la$ on $\r Z_\el$ we have
\be{4d}\FL_\el^\ln(t)=\f{\ln^{d_\el}}{\la^{d_\el}}
\S_{\m\in\Nl^\el_+}\f{\lG_\el(\m+d'_\el/\el)}{\lG_\el(\m+d_\el/\el)}\ \f{\lG(d_\el+|\m|)}{\lG(d_\el+|\m|/\la)}\ E_c^\m(\ln^{1/\la}t)\ee
for $t\in\lO_c,$ simplifying to
\be{4n}\FL_\el^\ln(t)=\ln^{d_\el}\f{\lG_\el(d'_\el/\el)}{\lG_\el(d_\el/\el)}\ {}_1F_1\Hy{d'_\el/\el}{d_\el/\el}{\ln t}\ee
for $\la=1,$ and (ii) for the potential $\lf_\el(w):=\log\lD(w,w)^{-p}$ on $\r Z_\el\ui\c Z,d\lr)$ we have
\be{4o}\FL_\el^\ln(t)=\f{\lG_\el(\ln)\ \lG_\el(d'_\el/\el)}{\lG_\el(\ln-d_\el/\el)\ \lG_\el(d_\el/\el)}
\ {}_2F_1\Hy{\f{d'_\el}\el;\ln}{\f{d_\el}\el}t\ee
for $t\in\lO_c\ui(c-\lO_c).$
\end{corollary}
\begin{proof} In view of \er{4h}, the assertions follow with \er{3p} and \er{3q}, respectively.
\end{proof}

\begin{corollary}\label{4p} In the rank $\el=1$ setting the reproducing kernel satisfies
$$\KL_1^\ln(z,w)=\f{\lG(a/2)}{\lG(d/r)\ \lG(ar/2)}\ t^{\f{1-r}2a}\ \dl_t^b\ t^{\f{r-1}2a+b}\P_{j=2}^r\f{D_j t+tD_j}2\FL_1^\ln(t)\big|_{t=(z|w)}$$
where (i) for the potential $\lf_1(w)=(w|w)^\la$ on $\r Z_1$ we have
\be{4r}\FL_1^\ln(t)=\f{\ln^{p-1}}{\la^{p-1}}\ E_{\f1\la,p-1}(\ln^{1/\la}t)=\f{\ln^{p-1}}{\la^{p-1}}\S_{m\ge 0}\f{\ln^{m/\la}t^m}{\lG(p-1+m/\la)}\ee
for $t>0,$ where $E_{A,B}(t):=\S_{m=0}^\oo\f{t^m}{\lG(Am+B)}$ is the generalized {\bf Mittag-Leffler function}, and (ii) for the potential $\lf_1(w)=\lD(w,w)^{-p}$ on $\r Z_1\ui\c Z$ we have
\be{4t}\FL_1^\ln(t)=\f1{\lG(\ln+1-p)}\S_{m\ge0}\f{\lG(\ln+m)}{\lG(p-1+m)}\ t^m\ee
for $0<t<1.$
\end{corollary}
\begin{proof} In view of \er{3e}, \er{4d} simplifies to \er{4r}, since $d_1=p-1$ and $d'_1=1.$ Similarly, \er{4o} simplifies to \er{4t}.
\end{proof}

\section{Asymptotic expansion of Bergman kernels}
For a given K\"ahler potential $\lf,$ the function $e^{-\ln\lf(w)}\KL^\ln(w,w)$ is closely related to the {\bf Kempf distortion function} (or Rawnsley's $\Le$-function) which is of importance in the study of projective embeddings and constant scalar curvature metrics (Donaldson~\cite{Do}), where a prominent role is played by the asymptotic behavior as $\ln\to+\oo$ (sometimes referred~to as Tian-Yau-Zelditch (TYZ) expansion). For more details, see \cite{E}. In this section, we prove our major results concerning the asymptotic expansion.

\begin{theorem} Let $\el=1$ and $d\lr(t)=\la e^{-\ln t^\la} t^{\la(p-1)-1}\,dt.$ Then the reproducing kernel of the associated Hilbert space $\HL_\ln:=\HL_\lr$ of holomorphic functions on $\r Z_1$ satisfies
$$\f{\lG(d/r)\ \lG(ar/2)}{\lG(a/2)\ \ln^{p-1}}\ e^{-\ln(z|z)^\la}\ \KL_1^\ln(z,z)=\S_{j=0}^{p-2}\f{b_j}{\ln^j\ (z|z)^{j\la}}+O(e^{-\lh\ln(z|z)^\la})$$
as $\ln\to+\oo$, with some $\lh>0$ and constants $b_j$ independent of $z$ and~$\ln$. Furthermore, $b_0=1.$
\end{theorem}
\begin{proof} Put $s=\ln^{1/\la}t.$ As $s\to+\oo$, it~is known that
$$E_{A,B}(s)=\f1A s^{(1-B)/A} e^{s^{1/A}}+O(e^{(1-\lh)s^{1/A}})$$
with some $\lh>0$, see~\cite{E} (one~can take any $0<\lh<\f1{\F2}(1-\cos\f{2\pi}A)$ for $A\ge4$, and any $0<\lh<\f1{\F2}$
for $0<A\le4$); furthermore, this remains valid if a derivative of any order is applied to the left-hand side and to the first term on the right-hand side. In~our case, $1/A=\la$ and $(1-B)=2-p$. For any polynomial $P(s)$ and constants $\lg,c\in\Rl$, a simple induction argument shows that
$$\dl_s^k s^{\lg\la+c}P(s^\la)e^{s^\la}=s^{\lg\la+c-k}Q(s^\la)e^{s^\la}$$
with another polynomial~$Q$ of degree $\deg P+k,$ whose leading coefficient equals $\la^k$ times the leading coefficient of~$P$.
It~follows that $\f{sD_j+D_js}2\ s^{\lg\la}P(s^\la)e^{s^\la}=s^{\lg\la} Q(s^\la)e^{s^\la},$ with $\deg Q=\deg P+a$ and leading
$(Q)=$leading$(P)\cdot\la^a$; and similarly for the differential operator~$s^{\f{1-r}2a}\dl_s^b s^{\f{r-1}2a+b}$. Altogether we get
$$s^{\f{1-r}2a}\dl_s^b s^{\f{r-1}2a+b}\P_{j=2}^r\f{sD_j+D_js}2\ s^{\lg\la}e^{s^\la}=s^{\lg\la}Q(s^\la)e^{s^\la},$$
where $Q$ has degree $(r-1)a+b=p-2$ and leading coefficient $\la^{(r-1)a+b}=\la^{p-2}$. Taking $\lg=2-p$ and using Corollary~\ref{4p} gives the result.
\end{proof}

The asymptotic expansion as $\ln\to+\oo$ (with the other parameters fixed) of the kernels from Corollary~\ref{4m} can~also be obtained by standard methods. As our main results, we will now derive the asymptotic expansion of the reproducing kernel functions associated with the K\"ahler potentials $\lf_\el(w)$ from Proposition \ref{2h}, both in the flat and the bounded setting. By Corollary \ref{4m} this is closely related to the asymptotic expansion of multi-variable hypergeometric functions, which by itself is a difficult matter. Accordingly, the following two theorems are of independent interest.

\begin{theorem}\label{5a} For $\Re\la>\f dr-1$ and $\Re\lb>\f dr-1$, there is an asymptotic expansion
\be{5b}\f{\lG_r(\la)}{\lG_r(\lb)}{}_1F_1\Hy\la\lb z\approx\f{e^{(z|e)}}{N(z)^{\lb-\la}}\ {}_2F_0\Hy{\f dr-\la,\lb-\la}{-}{z^{-1}}\ee
as $|z|\to+\oo$ while $\f z{|z|}$ stays in a compact subset of~$\lO$. In particular,
\be{5u}\f{\lG_r(d/r)}{\lG_r(\lb)}{}_1F_1\Hy{d/r}\lb z\approx\f{e^{(z|e)}}{N(z)^{\lb-d/r}}.\ee
These expansions can be differentiated termwise any number of times.
\end{theorem}
Here as usual \er{5b} means that for any $k=0,1,2,\dots$ we have $\Big|\f{\lG_r(\la)}{\lG_r(\lb)}\f{N(z)^{\lb-\la}}{e^{(z|e)}}
{}_1F_1\Hy\la\lb z-\S_{|\m|\le k}\Big|=O(|z|^{-k-1})$ as $|z|\to\oo$, when ${}_2F_0$ is expanded as a formal (nowhere convergent) power series according to \er{4l}.
\begin{proof} For any $z\in\lO,$ the element $g=P_z^{1/2}\in\operatorname{Aut}(\lO)$ satisfies $z=ge$ and $g=g^*.$ Hence $(gx|e)=(x|ge)=(x|z).$ Let
$\Re\la>\f dr-1,\ \Re(\lb-\la)>\f dr-1.$ By \cite[Proposition~XI.1.4]{FK} we~have the integral representation
\be{4v}\f{\lG_r(\la)}{\lG_r(\lb)}\ \jfj\la\lb z=\f1{\lG_r(\lb-\la)}\I_{\lO\ui(e-\lO)}dx\ e^{(gx|e)}\ N(x)^{\la-d/r}
\ N(e-x)^{\lb-\la-d/r}$$
$$=\f1{\lG_r(\lb-\la)}\I_{\lO\ui(e-\lO)}dx\ e^{(x|z)}\ N(x)^{\la-d/r}\ N(e-x)^{\lb-\la-d/r}.\ee
(The~result is stated there only for $z\in\lO\ui(e-\lO)$ but with general hypergeometric function ${}_p\!F_q$; for $p=q$ as~we have here, the~proof there actually works without changes for all $z\in\lO$.) Making the change of variable $x\mapsto e-x$ in \er{4v} gives
\be{5c}\f{\lG_r(\la)}{\lG_r(\lb)}\ \jfj\la\lb z=\f{e^{(z|e)}}{\lG_r(\lb-\la)}\I_{\lO\ui(e-\lO)}dx\ e^{-(x|z)}\ N(e-x)^{\la-d/r}\ N(x)^{\lb-\la-d/r}.\ee
The rest of the proof is divided into two parts. In the first part we prove the assertion under the assumption $\Re(\lb-\la)\ge\f dr.$ In this case the~contribution to the last integral from the complement of any neighbourhood of zero is exponentially small as $|z|\to+\oo$ as~above. Indeed, if~$\f z{|z|}$ stays in a compact subset of~$\lO$, then $(x|z)\ge c|x||z|$ for all $x\in\lO$ with some $c>0$ \cite[I.1.5]{FK}, so~the integral over $|x|>\ld$ is $O(e^{-c\ld|z|})$ as $|z|\to+\oo.$ This implies \er{5u} since, up to an error of $O(e^{-c\ld|z|})$, we have
$$\f{\lG_r(d/r)}{\lG_r(\lb)}\ \jfj{d/r}\lb z\approx\f{e^{(z|e)}}{\lG_r(\lb-d/r)}\I_{\lO}dx\ e^{-(x|z)}\ N(x)^{\lb-2d/r}=e^{(z|e)}\ N(z)^{d/r-\lb}.$$
In order to show \er{5b}, consider the Faraut-Koranyi `binomial theorem' \cite[XII.1.3]{FK}
\be{5d}N(e-x)^{\la-\f dr}={}_1F_0\Hy{\f dr-\la}{-}x=\S_\m\(\f dr-\la\)_\m\ E^\m(x,e)\quad\text{uniformly}.\ee
For $|x|\le\ld,$ with $\ld<1,$ the~remainder $\S_{|\m|\ge k}$ of the last series is $\le C_k|x|^k$ in modulus, for~each $k=0,1,\dots$, with some finite constant $C_k$. Hence
$$\I_{\lO\ui\{|x|<\ld\}}dx\ e^{-(x|z)}\ |x|^k\ N(x)^{\lb-\la-d/r}\le\I_{\lO\ui\{|x|<\ld\}}dx\ e^{-c|x||z|}|x|^{k+r(\lb-\la)-d}
\le\I_{X}dx\ e^{-c|x||z|}|x|^{k+r(\lb-\la)-d}$$
$$=vol(\Sl^{d-1})\I_0^\oo dt\ e^{-ct|z|}\ t^{k+r(\lb-\la)-1}=\f{vol(\Sl^{d-1})\ \lG (k+r(\lb-\la))}{(c|z|)^{k+r(\lb-\la)}}
=O(|z|^{-k-r(\lb-\la)}).$$
Substituting \er{5d} into \er{5c} we~thus get, for~each $k=0,1,\dots$,
\be{5e}\begin{aligned}&\f1{\lG_r(\lb-\la)}\I_{\lO\ui\{|x|<\ld\}}dx\ e^{-(x|z)}\ N(e-x)^{\la-d/r}\ N(x)^{\lb-\la-d/r}\\&\hskip4em=\S_{|\m|<k}(d/r-\la)_\m\f1{\lG_r(\lb-\la)}\I_{\lO\ui\{|x|<\ld\}}dx\ e^{-(x|z)}\ E_e^\m(x)\ N(x)^{\lb-\la-\f dr}+O(|z|^{-k-r(\lb-\la)}).\end{aligned}\ee
The integral term, again up to an error of $O(e^{-c\ld|z|})$, equals
$$\f1{\lG_r(\lb-\la)}\I_{\lO}dx\ e^{-(x|z)}\ E_e^\m(x)\ N(x)^{\lb-\la-\f dr}=(\lb-\la)_\m\ E_e^\m(z^{-1})\ N(z)^{\la-\lb}$$
by~\cite[XI.2.3]{FK}. Inserting this into \er{5e}, we~get~\er{5b}. Furthermore, standard arguments used in the stationary phase method
(which is essentially what the proof above was about) show that \er{5b} (and \er{5u}) can be differentiated termwise any number of times. We~have thus established the theorem under the hypothesis $\Re(\lb-\la)\ge\f dr.$\\
The second part of the proof removes this restriction via a shifting argument in the parameter $\lb.$ Suppose $(\lb-1)_1\ne 0.$ The elementary relation $(\lb-1)_\m\ \<\lb-1\>^\m=(\lb-1)_1\ (\lb)_\m$ together with $D_{\lb-\f dr}E_e^\m=\<\lb-1\>^\m\ E_e^\m$
(cf. \er{4u}) implies
\be{4s}\f{\lG_r(\la)}{\lG_r(\lb-1)}\ \jfj\la{\lb-1}z=\f{\lG_r(\la)}{\lG_r(\lb)}\ (\lb-1)_1\ \jfj\la{\lb-1}z
=\f{\lG_r(\la)}{\lG_r(\lb)}\ D_{\lb-\f dr}\ \jfj\la\lb z\ee
In the {\bf special case} $\la=d/r$ assume that the expansion \er{5u} holds for parameter $\lb.$ Then
$$\f{\lG_r(d/r)}{\lG_r(\lb-1)}\ \jfj{d/r}{\lb-1}z=\f{\lG_r(d/r)}{\lG_r(\lb)}\ D_{\lb-\f dr}\jfj{d/r}\lb z$$
$$\approx D_{\lb-\f dr}\(\f{e^{(z|e)}}{N(z)^{\lb-d/r}}\)=N(z)^{1-\lb+d/r}\ \dl_N e^{(z|e)}=\f{e^{(z|e)}}{N(z)^{\lb-1-d/r}}$$
since $N(e)=1.$ Thus \er{5u} holds also for $\lb-1$ instead of $\lb.$\\
The {\bf general case} requires more effort. Suppose $m_j>\f a2(j-r)+\Re\la-1.$ Then \cite[Lemma XI.2.3]{FK} implies
\be{5f}\I_{\lO}dx\ e^{-(x|z)}\ N(x)^{-\la}\ E_e^\m(x)=\lG_r(\m-\la+d/r)\ N(z)^{\la-d/r}\ E_e^\m(z^{-1}).\ee
The $L$-invariant polynomial $N(e-x)\ E_e^\m(x)$ has a (finite) Peter-Weyl decomposition
$$N(e-x)\ E_e^\m(x)=\S_\n C_\n^\m\ E_e^\n(x)$$
with unique coefficients $C_\n^\m$ related to the so-called Pieri rules \cite[Section 4]{S} (thus $C_\n^\m\ne 0$ implies $\m\ic\n$ and $|\n|\le|\m|+r$). For any $\lg$ the calculation
$$\S_\n(\lg-1)_\n\ E_e^\n(z)=\lD(z,e)^{1-\lg}=N(e-z)^{1-\lg}=N(e-z)\ \lD(z,e)^{-\lg}$$
$$=N(e-z)\S_\m(\lg)_\m\ E_e^\m(z)=\S_\m(\lg)_\m\S_\n C_\n^\m\ E_e^\n(z)=\S_\n E_e^\n(z)\S_\m C_\n^\m\ (\lg)_\m$$
implies
\be{5i}(\lg-1)_\n=\S_\m C_\n^\m\ (\lg)_\m\ee
for all $\n\in\Nl^r_+.$ We claim that for each $\m$ and all $\la\in\Cl$ we have
\be{5g}\(\f dr-\la\)_\m\ \dl_N\(e^{(z|e)}\ N(z)^{\la-d/r}\ E_e^\m(z^{-1})\)
=e^{(z|e)}\ N(z)^{\la-d/r}\S_\n C_\n^\m\ \(\f dr-\la\)_\n\ E_e^\n(z^{-1}).\ee
In fact for all values of $\la$ for which the various integrals above converge we have
$$\lG_r(\m-\la+d/r)\ \dl_N\(e^{(z|e)}\ N(z)^{\la-d/r}\ E_e^\m(z^{-1})\)=\dl_N\(e^{(z|e)}\I_{\lO}dx\ e^{-(x|z)}\ N(x)^{-\la}\ E_e^\m(x)\)$$
$$=\dl_N\(\I_{\lO}dx\ e^{(e-x|z)}\ N(x)^{-\la}\ E_e^\m(x)\)
=\I_{\lO}dx\ N(e-x)\ e^{(e-x|z)}\ N(x)^{-\la}\ E_e^\m(x)$$
$$=e^{(z|e)}\I_{\lO}dx\ e^{-(x|z)}\ N(x)^{-\la}\ N(e-x)\ E_e^\m(x)
=e^{(z|e)}\S_\n C_\n^\m\I_{\lO}dx\ e^{-(x|z)}\ N(x)^{-\la}\ E_e^\n(x)$$
$$=e^{(z|e)}\ N(z)^{\la-d/r}\S_\n C_\n^\m\ \lG_r(\n-\la+d/r)\ E_e^\n(z^{-1}),$$
using \er{5f} with $n_j\ge m_j$ in the last step. Since both sides of \er{5g} are entire functions of $\la,$ \er{5g} holds in general by analytic continuation. Now assume that \er{5b} holds for some~$\lb$ with $(\lb-1)_1\ne0.$ With \er{5g} and \er{5i}, applied
to $\lg:=\lb-\la,$ we obtain
$$\f{\lG_r(\la)}{\lG_r(\lb-1)}\ {}_1F_1\Hy\la{\lb-1}z=\f{\lG_r(\la)}{\lG_r(\lb)}\ (\lb-1)_1\ {}_1F_1\Hy\la{\lb-1}z
=\f{\lG_r(\la)}{\lG_r(\lb)}\ D_{\lb-\f dr}\ {}_1F_1\Hy\la\lb z$$
$$\approx\S_\m\(\f dr-\la\)_\m(\lb-\la)_\m D_{\lb-\f dr}\(\f{e^{(z|e)}}{N(z)^{\lb-\la}}\ E_e^\m(z^{-1})\)$$
$$=\S_\m(\lb-\la)_\m\ N(z)^{1-\lb+d/r}\ \(\f dr-\la\)_\m\ \dl_N\(e^{(z|e)}\ N(z)^{\la-d/r}\ E_e^\m(z^{-1})\)$$
$$=\S_\m(\lb-\la)_\m\ \f{e^{(z|e)}}{\ N(z)^{\lb-1-\la}}\S_\n C_\n^\m\(\f dr-\la\)_\n\ E_e^\n(z^{-1})$$
$$=\f{e^{(z|e)}}{N(z)^{\lb-1-\la}}\S_\n\(\f dr-\la\)_\n\ E_e^\n(z^{-1})\S_\m\ C_\n^\m(\lb-\la)_\m
=\f{e^{(z|e)}}{N(z)^{\lb-1-\la}}\S_\n\(\f dr-\la\)_\n(\lb-1-\la)_\n\ E_e^\n(z^{-1}).$$
Thus \er{5b}~also holds for $\lb-1$ in place of $\lb.$ By~induction, this proves \er{5b} even for all complex $\lb$ such that $(\lb)_\m\ne0$ $\forall\m$ (which includes, in~particular, all~$\lb$ with $\Re\lb>\f dr-1$).
\end{proof}

%The~sum on the right-hand side being finite (since it involves only $\m$ with $|\m|\le|\n|+r$), it~follows by analytic continuation that \er{5i} remains in force not only for $\Re(\lg-\lm)\ge\f dr+1$, but in fact for all complex~$\lg$.

Using Kummer's relation
\be{5j}\jfj\la\lb z=e^{(z|e)} \jfj{\lb-\la}\lb{-z}\ee
(which can be gleaned for $z\in\lO$ from \er{4v} by making the change of variable $x\mapsto e-x$, and for general $z$ by analytic continuation), it~is even possible to relax also the condition on~$\la$ in Theorem~\ref{5a}; the~theorem thus holds in fact for all
$\la,\lb$ such that $(\la)_\m(\lb)_\m\neq0\ \forall\m$. We~omit the details.

%\By~analogy with the classical formula for the Segal-Bargmann-Fock space on~$\Cl^n$, one~wonders whether $2^{\el q}|S_\el|\lG_\el(\f{d'_\el}\el)C_\el=\pi^n$, i.e.~$|S_\el|=2^{-\el q}\pi^n C_\el^{-1}/\lG_\el(\f{d'_\el}\el)$.

% In~particular, for $\el=1$ we recover the results from the last section.
% As an application we obtain
% $$\la\I_{\lO_c}dx\ e^{-\ln (x|e)^\la} (x|e)^{\la\g d'_\el-1} N_\m(x)
% =\f{\lG(\g d'_\el+\f{|\m|+d'_\el-1}\la)}{\ln^{\g d'_\el+\f{|\m|+d'_\el-1}\la}}\ \f{\lG_{\lO_c}(\m+d'_\el/\el)}{\lG(|\m|+d'_\el)}$$
% $$=\lG_{\lO_c}\(\f{d'_\el}\el\)\f{\lG(\g d'_\el+\f{|\m|+d'_\el-1}\la)}{\ln^{\g d'_\el+\f{|\m|+d'_\el-1}\la}\ \lG(|\m|+d'_\el)}
% \ \(\f{d'_\el}\el\)_\m$$
% $$=(2\lp)^{\el(\el-1)a/4}\P_{j=1}^\el\ \lG(1+\f a2(\el-j))\f{\lG(\g d'_\el+\f{|\m|+d'_\el-1}\la)}{\ln^{\g d'_\el+\f{|\m|+d'_\el-1}\la}
% \ \lG(|\m|+d'_\el)}\ \(\f{d'_\el}\el\)_\m$$
% where
% $\f{d'_\el}{\el}=1+\f a2(\el-1)$.

Our second main result in this section gives the asymptotic expansion for the multi-variable Gauss hypergeometric function ${}_2F_1,$ for a~euclidean Jordan algebra $X$ of dimension $d$ and rank~$r$.

\begin{theorem}\label{5l} Let $\la,\lb>\f dr-1$ and $y\in\lO\ui(e-\lO)$. Then there is an asymptotic expansion
\be{5n}\f{\lG_r(\la)\ \lG_r(\ln)}{\lG_r(\lb)}{}_2F_1\Hy{\la,\ln}\lb y\approx\lG_r(\la-\lb+\ln)
\ \f{N(y)^{\la-\lb}}{N(e-y)^{\la-\lb+\ln}}\ {}_2F_1\Hy{\f dr-\la,\lb-\la}{\f dr+\lb-\la-\ln}{e-y^{-1}}\ee
as $\ln\to+\oo.$ In particular
\be{5v}\f{\lG_r(d/r)\ \lG_r(\ln)}{\lG_r(\lb)}{}_2F_1\Hy{d/r,\ln}\lb y\approx\lG_r(d/r-\lb+\ln)\ \f{N(y)^{d/r-\lb}}{N(e-y)^{d/r-\lb+\ln}}.\ee
These expansions can be differentiated any number of times.
\end{theorem}
\begin{proof} For $\ln>\f dr-1$, we have the integral representation \cite[XV.1.3]{FK}
\be{5o}\lG_r(\ln)\ {}_2F_1\Hy{\la,\ln}\lb y=N(y)^{-\ln}\I_{\lO}dx\ e^{-(x|y^{-1})}{}_1F_1\Hy\la\lb x\ N(x)^{\ln-\f dr}.\ee
Since $y\in\lO\ui(e-\lO)$ it follows that $z:=y^{-1}-e\in\lO.$ The hypothesis $\la,\lb>\f dr-1$ ensures that $(\la)_m,(\lb)_\m>0\ \forall\m$ and also $(\la)_\m$ is a nondecreasing function of $\la$ for each~$\m$. Choose an integer $M\ge0$ so that $\la\le\lb+M.$
Kummer's relation \er{5j} and the fact that $\jfj{-M}\lb x$ is a polynomial of degree~$M$ imply
\be{5p}\jfj\la\lb x\le\jfj{\lb+M}\lb x=e^{(x|e)}\jfj{-M}\lb x\le c_M (x|z)^M e^{(x|e)}\ee
for some constant $c_M>0.$ The expression $\ld_z(x):=\f{N(x)}{(x|z)^r}$ defines a positive continuous function on~$\lO\setminus\{0\}$, which is homogeneous of degree~0 and vanishes at the boundary $\dl\lO\setminus\{0\}.$ Put $U_\ld:=\{x\in\lO:\ \ld_z(x)<\ld\}.$ Write
$x=t\lz$, where $t=(x|z)>0$ while $\lz\in\(\f z{(z|z)}\op\ z^\perp\)\ui\lO=:\lO_z$ satisfies $(\lz|z)=1.$ With $dx=t^{d-1}dt\ d\lz$ we obtain
$$\f1{c_M}\I_{U_\ld}dx\ e^{-(x|y^{-1})}\ {}_1F_1\Hy\la\lb x\ N(x)^{\ln-\f dr}\le\I_{U_\ld}dx\ e^{-(x|z)}\ (x|z)^{M+r\ln-d}\ N(x)^{\ln-\f dr}$$
$$=\ld^{\ln-d/r}\I_{U_\ld}dx\ e^{-(x|z)}\ (x|z)^{M+r\ln-d}=\ld^{\ln-d/r}\ |U_\ld\ui\lO_z|\I_0^\oo dt\ t^{d-1}\ e^{-t}\ t^{M+r\ln-d}
=|U_\ld\ui\lO_y|\ \lG(M+r\ln)\ \ld^{\ln-\f dr}.$$
Since $\f{\lG(M+r\ln)}{\lG_r(\ln)}\sim(r\ln)^M\ln^{\f{r(r-1)}2a}\f{(2\pi)^{\f{1-r}2}}{\F r} r^{r\ln} \ln^{\f{r-1}2}$ by~Stirling's formula, we~get
$$\f{N(y)^{-\ln}}{\lG_r(\ln)}\I_{U_\ld}dx\ e^{-(x|y^{-1})}{}_1F_1\Hy\la\lb x\ N(x)^{\ln-\f dr}
\le c'\(\f{r^r\ld}{N(y)}\)^\ln\ln^{M+\f{r-1}2-\f{r(r-1)}2a}.$$
For $\ld<\f{N(y)}{r^r N(e-y)}$, this is exponentially smaller as $\ln\to+\oo$ than the right-hand side of~\er{5m}. It follows that the contribution to \er{5o} from the integral over $U_\ld$, for $\ld>0$ small enough, is~negligible compared to the right-hand side of
\er{5m}. It~is therefore enough to integrate in \er{5o} over $\lO\setminus U_\ld$, where however it is legitimate to replace ${}_1F_1$ by its asymptotic expansion \er{5b} obtained in Theorem~\ref{5a}. By~the same argument as above the integrations over $\lO\setminus\lO(\ld)$ can be replaced by integrations over $\lO$ up to an exponentially small error.\\
In the {\bf special case} $\la=d/r$ this implies \er{5v} since
$$N(y)^\ln\f{\lG_r(d/r)\ \lG_r(\ln)}{\lG_r(\lb)}{}_2F_1\Hy{d/r,\ln}\lb y\approx\I_{\lO}dx\ e^{-(x|y^{-1})}
\ \f{e^{(x|e)}}{N(x)^{\lb-d/r}}\ N(x)^{\ln-\f dr}=\I_{\lO}dx\ e^{-(x|y^{-1}-e)}\ N(x)^{\ln-\lb}$$
$$=\lG_r(\ln-\lb+d/r)\ N(y^{-1}-e)^{\lb-\ln-d/r}=\lG_r(\ln-\lb+d/r)\f{N(y)^{\ln-\lb+d/r}}{N(e-y)^{\ln-\lb+d/r}}$$
using \cite[Proposition VII.1.2]{FK}, and
\be{5x}N(y^{-1}-e)=N(y^{-1})N(e-y).\ee
The {\bf general case} requires more effort. Put $\m^*=(m_r,\dots,m_2,m_1).$ Then $\lF_\m(x^{-1})=\lF_{-\m^*}(x)=\lF_{m_1-\m^*}(x)N(x)^{-m_1}$ \cite[p.~258]{FK} and $\lG_r(\lg-\m^*)=(-1)^{|\m|}\ \f{\lG_r(\lg)}{(\f dr-\lg)_\m}.$ By \cite[XI.2.3]{FK} the spherical polynomial $\lF_\m$ satisfies
\be{5w}\I_{\lO}dx\ e^{-(x|z)}\lF_\m(x^{-1})\ N(x)^{\lg-\f dr}=\I_{\lO}dx\ e^{-(x|z)}\lF_{m_1-\m^*}(x) N(x)^{\lg-m_1-\f dr}$$
$$=\lG_r(m_1-\m^*+\lg-m_1)\ \f{\lF_{m_1-\m^*}(z)}{N(z)^{\lg-m_1}}
=\lG_r(\lg-\m^*)\ \f{\lF_\m(z)}{N(z)^\lg}=\f{\lG_r(\lg)}{(\f dr-\lg)_\m}\ \f{\lF_\m(-z)}{\ N(z)^\lg}.\ee
Putting $\lg:=\la-\lb+\ln$ and using \er{5w} for $E_e^\m$ (which is proportional to $\lF_\m$), this yields
\be{5m}\f{\lG_r(\la)}{\lG_r(\lb)}{}_2F_1\Hy{\la,\ln}\lb y\approx\f{N(y)^{-\ln}}{\lG_r(\ln)}\I_{\lO}dx\ e^{-(x|y^{-1}-e)}\ N(x)^{\la-\lb}
\S_\m(\f dr-\la)_\m\ (\lb-\la)_\m\ E_e^\m(x^{-1})\ N(x)^{\ln-\f dr}$$
$$=\f{N(y)^{-\ln}}{\lG_r(\ln)}\S_\m(\f dr-\la)_\m\ (\lb-\la)_\m\I_{\lO}dx\ e^{-(x|y^{-1}-e)}\ N(x)^{\la-\lb+\ln-\f dr}\ E_e^\m(x^{-1})$$
$$=\f{N(y)^{-\ln}}{\lG_r(\ln)}\S_\m(\f dr-\la)_\m\ (\lb-\la)_\m\ \f{\lG_r(\la-\lb+\ln)}{(\f dr+\lb-\la-\ln)_\m}
\ N(y^{-1}-e)^{\lb-\la-\ln}\ E_e^\m(e-y^{-1})$$
$$=\f{N(y)^{\la-\lb}}{N(e-y)^{\la-\lb+\ln}}\ \f{\lG_r(\la-\lb+\ln)}{\lG_r(\ln)}
\S_\m\f{(\f dr-\la)_\m\ (\lb-\la)_\m}{(\f dr+\lb-\la-\ln)_\m}\ E_e^\m(e-y^{-1})\ee
using \er{5x} again. Now \er{5n} follows.
\end{proof}

Spending a little more time on the estimate~\er{5p}, it~is again possible to extend the last theorem to all complex
$\la,\lb$ such that $(\la)_\m(\lb)_\m\neq0$ $\forall\m$. We~omit the details.

For $r=1$, so~that ${}_2F_1$ reduces to the ordinary Gauss hypergeometric function, it~is possible to give a different proof by using one of the relations among `Kummer's 24 integrals', namely \cite[2.9~(27)]{BE}
\be{5q}\begin{aligned}{}_2F_1\Hy{\la,\ln}\lb z &=\f{\lG(1-\ln)\ \lG (\lb)}{\lG(\la-\ln+1)\ \lG (\lb-\la)} z^{-\la}
 {}_2F_1\Hy{\la-\lb+1,\la}{\la-\ln+1}{\f1z}\\
 &+ \f{\lG(1-\ln)\ \lG (\lb)}{\lG(\la)\ \lG (\lb-\la-\ln-1)}\(1-\f1z\)^{\lb-\la-\ln} \(-\f1z\)^\ln
 {}_2F_1\Hy{\lb-\la,1-\la}{\lb-\la-\ln+1}{1-\f1z}, \end{aligned} \ee
and then arguing that the first summand is negligible while the second, upon some technical work to show that it yields an asymptotic expansion as $\ln\to+\oo$ even though $1-\f1z$ in general no longer lies in the unit disc (and making sense of the asymptotics of the Gamma functions on the negative half-axis), gives~\er{5n}. Although the analogues of Kummer's 24 integrals for the general case have been worked out by Koranyi \cite{Ko}, we~are not aware of anything like \er{5q} in the literature, and do not know if this line of proof is feasible.

As an application of the previous two theorems we obtain the TYZ-expansion of the reproducing kernels, both in the flat and the bounded settting.
\begin{corollary}\label{5r} (i) For $\la=1$, the kernel \er{4n} has the asymptotic expansion
$$e^{-\ln(t|c)}\ \KL^\ln(\F t,\F t)\approx\f{\ln^{d_\el}}{p_0(t)}\S_{j=0}^{d''_\el}\f{p_j(t)}{\ln^j}$$
as $\ln\to+\oo$, uniformly in compact subsets of $\lO_c.$ Here $p_j(t)$ are polynomials, with $p_0(t)=N(t)^{d''_\el}.$\\
(ii) The kernel \er{4o} has the asymptotic expansion
\be{5s}N_c(c-t)^\ln\ \KL^\ln(\F t,\F t)
\approx\f{\lG_\el(\ln-d''_\el/\el)\ \ln^{d''_\el}}{\lG_\el(\ln-d_\el/\el)\ p_0(t)}\S_{j=0}^{d''_\el}\f{p_j(t)}{\ln^j}\ee
as $\ln\to+\oo,$ uniformly on compact subsets of $\lO_c\ui(c-\lO_c).$ Here $p_j(t)$ are polynomials, with
$p_0(t)=N(c-t)^{d''_\el(1-1/\el)}\ N(t)^{d''_\el}=N(c-t)^{(\el-1)(a(r-\el)+b)}\ N(t)^{{\el(a(r-\el)+b)}}.$
\end{corollary}
\begin{proof} Put $\lf(t)=e^{(t|c)},\lq(t)=1$ in the flat case, and $\lf(t)=N_c(c-t)^{-1},\lq(t)=N_c(c-t)$ in the bounded case.
Let $\dl_u$ denote the partial derivative in direction $u\in X_c.$ By the Leibniz rule, for any polynomial $p(t)$ there exist polynomials $p_1(t),p_0(t)$ such that
$$\lf(t)^{-\ln}\dl_u\(\lf^\ln\lq^\lb N_c^\lg\ p\)(t)=\lq^{\lb-1}(t)\ N_c(t)^{\lg-1}(\ln p_1(t)+p_0(t)),$$
with highest term $p_1(t)$ given by $(t|u)\ N_c(t)\ p(t)$ and $(\dl_u N_c)(c-t)\ N_c(t)\ p(t),$ respectively. By induction, there exist polynomials $p_0(t),\ldots,p_k(t),$ such that
$$\lf(t)^{-\ln}\dl_{u_1}\ldots\dl_{u_k}\(\lf^\ln\lq^\lb N_c^\lg p\)(t)=\lq(t)^{\lb-k}\ N_c(t)^{\lg-k}\S_{j=0}^k\ln^j p_j(t),$$
with highest term $p_k(t)$ given by $(t|u_1)\ldots(t|u_k)\ N_c(t)^k p(t)$ and $(\dl_{u_1}N_c)(c-t)\ldots(\dl_{u_k}N_c)(c-t)\ N_c(t)^k p(t),$ respectively. Now let $p(t)$ be $L_c$-invariant. Then there exist $L_c$-invariant polynomials $q_0(t),\ldots,q_\el(t)$ such that
$$\lf(t)^{-\ln}\dl_{N_c}\(\lf^\ln\lq^\lb N_c^\lg p\)(t)=\lq(t)^{\lb-\el}(t)\ N_c(t)^{\lg-\el}\S_{j=0}^\el\ln^j q_j(t),$$
with highest term $q_\el(t)$ given by $N_c(c)N_c(t)^\el p(t)=N_c(t)^\el p(t)$ and
$N_c(\nabla_{c-t}N_c)N_c(t)^\el p(t)=N_c(c-t)^{\el-1}N_c(t)^\el p(t),$ respectively. Here we use $N_c(\nabla_{c-t}N_c)=N_c(N_c(c-t)(c-t)^{-1})=N_c(c-t)^\el N_c((c-t)^{-1})=N_c(c-t)^{\el-1}.$ Hence for each complex $\la$ there exist $L_c$-invariant polynomials $q_0^\la(c),\ldots,q_\el^\la(t)$ such that
\be{5y}\lf(t)^{-\ln}\ D_\la\(\lf^\ln\lq^\lb N_c^\lg p\)(t)=N_c(t)^{1-\la}\ \lf(t)^{-\ln}\ \dl_{N_c}\(\lf^\ln\lq^\lb N_c^{\lg+\la}p\)(t)
=\lq(t)^{\lb-\el}\ N_c(t)^{1+\lg-\el}\S_{j=0}^\el q_j^\la(t),\ee
with highest term $q_\el^\la(t)=q_\el(t)$ independent of $\la.$ According to \er{4q}, there is a factorization $\DL_\el=D_{\la_q}\ldots D_{\la_1}$ for suitable parameters $\la_1,\ldots,\la_q.$ Iterating \er{5y} we obtain, using obvious notation
\be{5t}\lf(t)^{-\ln}\DL_\el\(\lf(t)^\ln\lq(t)^\lb N_c(t)^\lg p(t)\)
=\lq(t)^{\lb-q\el}N_c(t)^{\lg-q(\el-1)}\S_{j_1,\ldots,j_q=1}^\el\ln^{j_1}\ldots\ln^{j_q}p_{j_1,j_2,\ldots,j_q}^{\la_1,\la_2,\ldots,\la_q}(t)$$
$$=\lq(t)^{\lb-q\el}N_c(t)^{\lg-q(\el-1)}\S_{j=0}^{q\el}\ln^j Q_j(t),\ee
with highest term $Q_{q\el(t)}=q_{\el,\ldots,\el}^{\la_1,\ldots,\la_q}(t)$ equal to $N_c(t)^{q\el}p(t)$ and
$N_c(c-t)^{q(\el-1)}N(t)^{q\el}p(t),$ respectively. Now set $p(t)=1$ and put $q:=(r-\el)a+b=d''_\el/\el.$ We apply Theorem~\ref{5a} to
$X_c$, with $\la=\f{d'_\el}\el$ and $\lb=\f{d_\el}\el=\la+q.$ In the flat case we obtain
$$e^{-\ln(t|c)}\ \KL^\ln(\F t,\F t)=\ln^{d_\el}\f{\lG_\el(d'_\el/\el)}{\lG_\el(d_\el/\el)}\ e^{-\ln(t|c)}\ \DL_\el
\ {}_1F_1\Hy{d'_\el/\el}{d_\el/\el}{\ln t}$$
$$=\ln^{d_\el}\ e^{-\ln(t|c)}\ \DL_\el\(\f{e^{\ln(t|c)}}{N_c(\ln t)^q}(1+O(e^{-\ld\ln|t|}))\)
\approx\ln^{d'_\el}\ e^{-\ln(t|c)}\ \DL_\el\(\f{e^{\ln(t|c)}}{N_c(t)^q}\)=\f{\ln^{d'_\el}}{N_c(t)^{q\el}}\S_{j=0}^{q\el}\ln^j\ Q_j(t)$$
as $\ln\to+\oo$, uniformly for $t$ in a compact subset of~$\lO_c$, with some $\ld>0$. In the bounded case we obtain
$$N_c(c-t)^\ln\ \KL^\ln(\F t,\F t)=\f{\lG_\el(d'_\el/\el)\ \lG_\el(\ln)}{\lG_\el(d_\el/\el)\ \lG_\el(\ln-d_\el/\el)}
\ N_c(c-t)^\ln\ \DL_\el\ {}_2F_1\Hy{d'_\el/\el,\ln}{d_\el/\el}t$$
$$\approx\f{\lG_\el(\ln-q)}{\lG_\el(\ln-d_\el/\el)}\ N_c(c-t)^\ln\ \DL_\el\(\f{N_c(c-t)^{q-\ln}}{N_c(t)^q}[1+O(\ld^{-\ln})]\)
=\f{\lG_\el(\ln-q)}{\lG_\el(\ln-d_\el/\el)}\f{N_c(c-t)^{q(1-\el)}}{N_c(t)^{q\el}}\S_{j=0}^{q\el}\ln^j\ Q_j(t)$$
as $\ln\to+\oo$, uniformly for $t$ in a compact subset of~$\lO_c\ui(c-\lO_c)$, with some $\ld>0$.
\end{proof}

Replacing $\f{\lG_\el(\ln-d''_\el/\el)}{\lG_\el(\ln-d_\el/\el)}$ by its asymptotic expansion $\ln^{d'_\el}-\f{d'_\el d_\el}\el\ln^{d'_\el-1}+\dots$ as $\ln\to+\oo$, one obtains from \er{5s} the TYZ expansion of $N_c(c-x)^\ln\ \KL^\ln(\F x,\F x)$ into decreasing powers of~$\ln$ with leading power $\ln^{d_\el},$ which, however, does not to terminate in general.

\section{Invariant measures and $n$-forms}
The Riemann measure $d\lL_\el$ on $\r Z_\el$ is $K$-invariant, but has no invariance properties with respect to $\r K$ or a transitive subgroup. We will now investigate the existence of invariant measures and holomorphic differential forms of top-degree. A~complex manifold $M$ of dimension $n$ has a trivial canonical bundle if there exists a nowhere vanishing holomorphic $n$-form $\lT$ on $M$. If
$M:=G/H$ for a complex Lie group $G$, with Lie algebra $\gL$, and a closed complex Lie subgroup $H$, with Lie algebra $\hL\ic\gL$, a holomorphic $n$-form $\lT$ is called $G$-invariant if the pull-back $g^*\lT=\lT$ for all $g\in G$. In other words,
$$\lT_{g(z)}(\lz_1,\ldots,\lz_n)=\lT_z(T_z(g)^{-1}\lz_1,\ldots,T_z(g)^{-1}\lz_n)$$
for all $g\in G,\ z\in M$ and holomorphic tangent vectors $\lz_i\in T_{g(z)}(M)$. Let $\Le:G\to M:=G/H$ be the evaluation map
$\Le(g):=g(o)$ at the base point $o=H\in M$, with differential $T_e(\Le):\gL\to T_o(M)$ for the tangent space $T_o(M)\al\gL/\hL$.
Every $h\in H$ defines a linear transformation $T_o(h)$ acting on $T_o(M)$. It is well-known \cite[Lemma 1.5 and Proposition 1.6]{H} that there exists a non-zero $G$-invariant holomorphic $n$-form $\lT$ on $G/H$ if and only if $\det T_o(h)=1$ for all $h\in H$. In this case the $(n,n)$-form $\lT\yi\o\lT$ induces (after multiplying by a constant) a $G$-invariant measure on $M$. More generally
\cite[Theorem 1.7]{H}, a non-zero invariant measure $\lM$ exists if and only if the weaker condition $|\det T_o(h)|=1$ holds for all
$h\in H$.

In this section we consider such questions for the Kepler manifold $M=\r Z_\el$ and a suitable subgroup $G\ic\r K$ acting transitively. If $Z$ is of tube type and $\el=r$ we take $G:=\r K$. In all other cases ($Z$ not of tube type or $\el<r$) $\r Z_\el$ is homogeneous under the closed subgroup $\r K_0$ generated by all commutators in $\r K,$ since the evaluation map from the Lie algebra $\r\kL_0$ to the tangent space $T_c(\r Z_\el)$ is surjective. The group $\r K_0$ is connected, being the closure of a countable union of connected sets centered at the identity. The group $\r K_1:=\{g\in\r K:\ |\det_Z g|=1\}$ contains both $\r K_0$ and $K.$

By \cite[Lemma 2.32]{Sch} every $g\in\r K$ satisfies
\be{6a}g Z_2^z=Z_2^{gz},\ g^*Z_0^{gz}=Z_0^z\ee
for all $z\in Z$, using the generalized Peirce spaces studied in \cite{Sch}. Define the subgroup
$$\r K^U:=\{g\in\r K:\ gU=U\}=\{g\in\r K:\ gc\in U\}\ic\r K,$$
with Lie algebra $\r\kL^U:=\{A\in\kL:\ AU\ic U\}$, and its subgroup
$$\r K^c:=\{h\in\r K:\ gc=c\}\ic\r K^U,$$
with Lie algebra $\r\kL^c:=\{A\in\kL:\ Ac=0\}$. A similar notation will be used for $\r K_0$.

\begin{lemma} Every $g\in\r K^U$ satisfies $gW^\perp=W^\perp$.
\end{lemma}
\begin{proof} For $g\in\r K^U$ we have $Z_2^{gc}=gZ_c^2=gU=U=Z_c^2$. Therefore \cite[Proposition 2.19]{Sch} implies that $c$ and $gc$ are Peirce equivalent. It follows that $W=Z_c^0=Z_0^{gc}=g^{-*}Z_c^0=g^{-*}W$. This implies the assertion.
\end{proof}

In particular every $h\in\r K^c\ic\r K^U$ satisfies $hW^\perp=W^\perp$. Since $T_c(\r Z_\el)=U\op V=W^\perp$ there exists a non-zero holomorphic $d_\el$-form on $\r Z_\el$ which is invariant under $\r K$ (resp., $\r K_0$) if and only if $\det\ h|_{W^\perp}=1$ for all
$h\in\r K^c$ (resp. $\r K^c_0$).

The case where $Z$ is not of tube type is uninteresting. In fact, for $\el=r$ the group $\r K_0$ acts transitively on $\r Z_r$ and the constant $d$-form $\lT=dz_1\yi\ldots\yi dz_d$ is $\r K_0$-invariant since $\det g=1$ for all $g\in\r K_0$. On the other hand, if $Z$ is not of tube type and $0<\el<r$ then there exists $A\in\r\kL_0$ with $Ac=0$ and $\tr A|_{W^\perp}\ne 0$. It follows that the range of
$\det h|_{W^\perp},$ with $h\in\r K^c_0,$ is all of $\Cl\setminus\{0\}$. Thus a non-zero $\r K_0$-invariant $d_\el$-form (or an invariant measure) on $\r Z_\el$ cannot exist.

From now on we assume that $Z$ is of tube type. Consider first the case $\el=r$. Then $\r Z_r=\r Z$ consists of all invertible elements in the Jordan algebra $Z$ with unit element $e$ and Jordan determinant $N$. Moreover, $U=W^\perp=Z$. For $A\in\r\kL$ we have
$${\tr}_Z A=\f dr(Ae|e)=\f p2(Ae|e).$$
If $p=2k$ is even, then ${\det}_Z g=N(ge)^k$ for all $g\in\r K$. This implies ${\det}_Z h=N(he)^k=N(e)^k=1$ for all $h\in\r K^e$. Thus there exists a non-zero $\r K$-invariant $d$-form $\lT$ on $\r Z$ which is given by
$$\lT=\f{dz_1\yi\ldots\yi dz_d}{N(z)^k}$$
for all $z\in\r Z$, since $N(gz)=N(z)\ N(ge)$ for all $g\in\r K$. On the other hand, if $p$ is odd, there is no holomorphic square-root of $N(z)^p$ and an invariant $d$-form does not exist. This occurs in two situations: Either $Z$ is a spin factor of odd dimension, or $Z=\Cl^{r\xx r}_{sym}$ with $r$ even. However, since $|{\det}_Z h|=|N(he)|^{p/2}=1$ for all $h\in\r K^e$, we always have an invariant positive measure $\lM$ on $\r Z$, which (up to a positive constant) is given by
$$d\lM(z)=\(\f i2\)^d\ \f{dz_1\yi d\o z_1\yi\ldots\yi dz_d\yi d\o z_d}{|N(z)|^p}.$$
If $p$ is even, the measure $\lM$ and the holomorphic $d$-form $\lT$ are related by $\lM=\(\f i2\)^d\ (-1)^{d(d-1)/2}\ \lT\yi\o\lT.$

Thus the most interesting case is when $Z$ is of tube type and $0<\el<r$. For $1\le i,j\le r$  let $\r\kL_{ij}$ denote the complex vector space generated by all linear transformations $D(x,y)$, where $x\in Z_{ik},\ y\in Z_{kj}$ and $1\le k\le r$ is arbitrary.

\begin{lemma}\label{6b} $\r\kL^U$ contains the spaces $\r\kL_{ij}$ if $i,j\le\el$, or $\el<i,j$, or $i\le\el<j$.
\end{lemma}
\begin{proof} Let $x\in Z_{ik},\ y\in Z_{kj}$ for some $k$. If $i,j\le\el$ then $x,y\in\begin{cases}U&k\le\el\\V&k>\el\end{cases}$. If
$i,j>\el$ then $x,y\in\begin{cases}V&k\le\el\\W&k>\el\end{cases}$. In both cases $x,y\in Z_\la^c$ for the same $\la=0,1,2$, and the Peirce multiplication rules \cite[Theorem 3.13]{L2} imply $\{x;y;U\}\ic\{Z_\la^c;Z_\la^c;U\}\ic U$. If $i\le\el<j$ then $x\in\begin{cases}U&k\le\el\\V&k>\el\end{cases}$ and $y\in\begin{cases}V&k\le\el\\W&k>\el\end{cases}$. Since $\{U;V;U\}=\{0\}=\{V;W;U\}$ the assertion follows in the third case.
\end{proof}

Consider the commuting vector fields $D(e_k,e_k),\ 1\le k\le r$. By \cite[Lemma 1.5]{U2} we have
\be{6c}[D(c_k,c_k),A]=(\ld_i^k-\ld_j^k)A\ee
for all $k$ and $A\in\r\kL_{ij}$. Denote by $\tL_{-1}\ic\kL$ the real span of $iD(e_k,e_k)$ for $1\le k\le r$, and let $\tL$ be a maximal abelian subalgebra of $\kL$ containing $\tL_{-1}$. By \cite[Lemma 1.2]{U2} there is a decomposition $\tL=\tL_{-1}\op\ \tL_1$, where $\tL_1=\{A\in\tL:\ Ae_k=0\ \forall k\}$. Let
$$\r\kL=\tL^\Cl\op\S_{\la}\r\kL_\la$$
denote the {\bf root decomposition} of $\r\kL$ under $\tL^\Cl.$ By \cite[Theorem 1.7]{U2}, we have
$$\S_i\r\kL_{ii}=\tL^\Cl\op\S_{\la|_{\tL_{-1}=0}}\r\kL_\la.$$
Together with \cite[Corollary 1.6]{U2} we obtain
$$\r\kL=\S_{i=1}^r\r\kL_{ii}\ \op\ {\S_{i\ne j}}^\op\ \r\kL_{ij}=\tL^\Cl\op\S_{\la|_{\tL_{-1}=0}}\r\kL_\la\op{\S_{i\ne j}}^\op\ \r\kL_{ij}.$$
It follows that every $A\in\r\kL$ has a unique decomposition
\be{6d}A=\lx+\lh+\S_{\la|_{\tL_{-1}=0}}A_\la+\S_{i\ne j}A_{ij},\ee
where $A_{ij}\in\r\kL_{ij}$, $\lh\in\tL_1^\Cl$ satisfies $\lh e_k=0$ for all $k$, $A_\la\in\r\kL_\la$ and
$$\lx=\S_k\ll_k\ D(c_k,c_k)\in\tL_{-1}^\Cl.$$

\begin{lemma}\label{6e} Let $A\in\r\kL^U$ have the decomposition \er{6d}. Then $A_{ij}\in\r\kL^U$ for all $i,j$.
\end{lemma}
\begin{proof} In view of Lemma~\ref{6b} we may assume that $A=\S_{j\le\el<i}A_{ij}.$ Applying \er{6c} yields
$$[D(c_k,c_k),A]=\S_{j\le\el<i}[D(c_k,c_k),A_{ij}]=\S_{j\le\el<i}(\ld_i^k-\ld_j^k)A_{ij}$$
for all $k$. For $u\in U$ we have $[D(c_k,c_k),A]u=D(c_k,c_k)Au-AD(c_k,c_k)u\in U$ since $AU\ic U$. It follows that
$$\S_{j\le\el<i}(\ld_i^k-\ld_j^k)A_{ij}u\in U$$
for all $k$. Choosing $k\le\el$ we obtain $-\S_{\el<i}A_{ik}u\in U.$ The components $A_{ik}u\in\S_{m\le\el}Z_{im}$ belong to pairwise orthogonal subspaces of $U^\perp$. Thus each $A_{ik}u=0$ and hence $A_{ik}\in\r\kL^U$. Choosing $\el<k$ we obtain
$\S_{j\le\el}A_{kj}u\in U.$ The components $A_{kj}u\in\S_{m\le\el}Z_{km}$ belong to pairwise orthogonal subspaces of $U^\perp$. Thus each $A_{kj}u=0$ and hence $A_{kj}\in\r\kL^U$.
\end{proof}

\begin{proposition}\label{6f} Let $Z$ be of tube type. Then we have for all $A\in\r\kL^U$
$${\tr}_{W^\perp}A=\f{p-a\el}{2+a(\el-1)}{\tr}_U A+\f{a\el}p{\tr}_Z A=\f{p-a\el}2(Ac|c)+\f{a\el}2(Ae|e).$$
\end{proposition}
\begin{proof} Consider the decomposition \er{6d} of $A$. For $\lh$ we have ${\tr}_Z\lh=\f dr(\lh e|e)=0$, ${\tr}_U\lh=\f{d_U}{\el}(\lh c|c)=0$ and ${\tr}_W\lh=\f{d_W}{r-\el}(\lh(e-c)|e-c)=0$. Hence also ${\tr}_V\lh=0$, showing that ${\tr}_{W^\perp}\lh=0$. For a fixed root $\la\in\lD$ vanishing on $\tL_{-1}$ choose $H\in\tL$ such that $\la(H)\ne 0$. Then $[H,A_\la]=\la(H)A_\la.$ Similarly,
$$[D(c_k,c_k),A_{ij}]=(\ld_i^k-\ld_j^k)A_{ij}$$
for all $k$. Putting $B=A_\la$ or $B=A_{ij}$ we have $B=[C,B]$ for some $C\in\tL^\Cl$. It follows that ${tr}_Z B=0$. Since
$B\in\r\kL^U$ by Lemma \ref{6e} it follows that $B$ is also a commutator in $\r\kL^U$, which acts on $W^\perp$. Thus
${\tr}_{W^\perp}B=0$. Since $C$ commutes with the Peirce projection $\lp$ onto $U$ and $BU\ic U$, the same argument shows ${\tr}_U B=0$.
Thus it remains to consider $\lx$. For $1\le i\le j\le r$ and $z\in Z_{ij}$ we have
$$\lx z=\S_k\ll_k\ D(c_k,c_k)z=\S_k\ll_k(\ld_i^k+\ld_j^k)z=(\ll_i+\ll_j)z.$$
Thus ${\tr}_{Z_{ij}}\lx=(\ll_i+\ll_j)\dim Z_{ij}.$ Put $\ll':=\ll_1+\ldots+\ll_\el$ and $\ll'':=\ll_{\el+1}+\ldots+\ll_r$. Then
$${\tr}_U\lx=\S_{i\le j\le\el}{\tr}_{Z_{ij}}\lx=a\S_{i<j\le\el}(\ll_i+\ll_j)+2\S_{j\le\el}\ll_j=(a(\el-1)+2)\ll',$$
$${\tr}_V\lx=\S_{i\le\el<j}{\tr}_{Z_{ij}}\lx=a\S_{i\le\el<j}(\ll_i+\ll_j)=a(\el\ll''+(r-\el)\ll'),$$
$${\tr}_W\lx=\S_{\el<i\le j}{\tr}_{Z_{ij}}\lx=a\S_{\el<i<j}(\ll_i+\ll_j)+2\S_{\el<j}\ll_j=(a(r-\el-1)+2)\ll''.$$
Hence ${\tr}_{W^\perp}\lx=p\ll'+a\el\ll''$ and ${\tr}_Z\lx=p(\ll'+\ll'')$. It follows that
$${\tr}_{W^\perp}\lx=\f{p-a\el}{2+a(\el-1)}{\tr}_U\lx+\f{a\el}p{\tr}_Z\lx.$$
\end{proof}

\begin{lemma}\label{6g} Let $s=(s',s'')\in\r U\xx\r W$ satisfy ${\det}_Z P_s=1.$ Then ${\det}_{W^\perp}P_s=N_c(s')^{p-a\el}.$
\end{lemma}
\begin{proof} For all $s=(s',s'')\in\r U\xx\r W$ the (Peirce diagonal) transformation $P_s$ satisfies
${\det}_Z P_s=N(P_s e)^{d/r}=N(s)^p=N_c(s')^p\ N_{e-c}(s'')^p.$ Since $W=Z_2^{e-c}$ has the genus $p-a\el$ we have ${\det}_W P_{s''}=N_{e-c}(s'')^{p-a\el}.$ Therefore
$${\det}_{W^\perp}P_s=\f{{\det}_Z P_s}{{\det}_W P_{s''}}=N_c(s')^p\ N_{e-c}(s'')^{a\el}.$$
\end{proof}

\begin{theorem} Let $Z$ be of tube type and $0<\el<r$. Then there exists a $\r K_0$-invariant holomorphic $d_\el$-form $\lT$ on
$\r Z_\el$ if and only if $p-a\el=2+a(r-\el-1)$ is even. (Among all tube type domains, $p-a\el$ is odd only for the symmetric matrices
$Z=\Cl^{r\xx r}_{sym}$ with $r-\el$ even.)
\end{theorem}
\begin{proof} For a diagonal element $s=\S_{j=1}^r s_j c_j$ with $s_1\cdots s_r=1$ consider the transformation $P_sz=Q_sQ_ez$ in
$\r K_0$. Put $s':=s_1\cdots s_\el,\ s'':=s_{\el+1}\cdots s_r$. For $s'=\pm 1$ we have $P_sc=c$. Since ${\det}_Z P_s=1$, Lemma~\ref{6g} implies ${\det}_{W^\perp}P_s=s'^{p-a\el}$. If $p-a\el$ is odd, this becomes $-1$ if $s'=-1$. Thus an invariant $d_\el$-form does not exist. Now assume that $p-a\el=2k$ is even. Since ${\det}_Z g=1$ for all $g\in\r K_0$ Proposition~\ref{6f} implies
$\tr_{W^\perp}A=k(Ac|c)$ for all $A\in \r\kL_0^U$. The Peirce manifold $M_\el\al\r K_0/\r K_0^U$ is the conformal compactification of $V=Z_c^1$ and is therefore simply-connected. The exact homotopy sequence implies that $\r K_0^U$ is connected. Therefore Lemma~\ref{6g} yields ${\det}_{W^\perp}g=N_c(gc)^k$ for all $g\in\r K_0^U$. In particular, ${\det}_{W^\perp}h=N_c(hc)^k=N_c(c)^k=1$ whenever $h\in\r K_c^0$, proving the existence of an invariant $d_\el$-form.
\end{proof}

\begin{proposition}\label{6h} Let $Z$ be of tube type. In case $p-a\el=2k$ is even, the invariant holomorphic $n$-form $\lT$ on
$\r Z_\el$ ($n:=d_\el$) has the local representation
\be{6i}(\ls^*\lT)_{u,v}=N_c(u)^{a(r-\el)-k}\ dz_1\yi\ldots\yi dz_n,\ee
where $a(r-\el)-k=\f a2(r+1-\el)-1.$
\end{proposition}
\begin{proof} Let $\lT_c=dz_1\yi\ldots\yi dz_n=du_1\yi\ldots\yi du_{d'_\el}\yi dv_1\yi\ldots\yi dv_{a\el(r-\el)}.$ Then, up to a constant factor, $\lT$ is given by $\lT_z:=\lT_c\oc\W^n g^{-1}$ on $\W^n T_z(\r Z_\el)=\W^n gW^\perp$, where $g\in\r K_0$ satisfies $gc=z$. Note that \er{6a} implies $g^*Z_0^z=g^*Z_0^{gc}=Z_c^0=W$ so that $gW^\perp={Z_0^z}^\perp=T_z(\r Z_\el)$. Let
$w:=\Lt(u,v)$. Then $Z_0^w=Z_0^{B_{v,-c}u}=B_{c,-v}^{-1}Z_0^u=B_{c,-v}^{-1}W$. For $u\in\r U$ choose $s=(s',s'')\in\r U\xx\r W$ such that $\det P_s=1$ and $P_sc=P_{s'}c=u$. Then $g:=B_{v,-c}P_s\in\r K_0$ satisfies $gc=B_{v,-c}u=z$. Therefore $gU=gZ_c^2=Z_2^{gc}=Z_2^w$ and $Z_0^w=Z_0^{gc}=g^{-1*}Z_c^0=g^{-1*}W$, and hence $gW^\perp=(Z_0^w)^\perp.$ Now Lemma~\ref{2f} implies
$g^{-1}\Lt'_{u,v}(\lx,\lh)=P_s^{-1}B_{v,-c}^{-1}\Lt'_{u,v}(\lx,\lh)=P_s^{-1}(\lx+D_{u,c}\lh).$ It follows that
$$(\ls^*\lT)_{u,v}=\lT_w\oc\W^n \Lt'_{u,v}=\lT_c\oc\W^n g^{-1}\Lt'_{u,v}=\lT_c\oc\W^n P_s^{-1}\oc(id_U\oplus D(u,c)|_V)
=\lT_c\ {\det}_{W^\perp}P_s^{-1}\ {\det}_V D_{u,c}.$$
Since ${\det}_{W^\perp}P_s=N_c(s')^{p-a\el}=N_c(u)^k$ by Lemma~\ref{6g}, the assertion follows with \er{2d}.
\end{proof}

In the example of the {\bf spin factor} $Z=\Cl^{n+1}$ of rank $r=2,$ for $\el=1$ it follows that $p-a\el=2+a(r-\el-1)=2$ is even. Thus there exists a non-zero holomorphic $n$-form $\lT$ on $\r Z_1$ which is invariant under $\r K_0=SO(n+1)$. This $n$-form is well-known
\cite[Lemma~2.2]{OPY} and has the form
$$\lT=\f{dz_1\yi\ldots\yi d z_n}{z_0}$$
on the open dense subset where $z_0\ne 0$. We have $U=\Cl c,\ W=\Cl\o c$ and $V=<c,\o c>^\perp$. For $v\in V$ we obtain
$Q_vc=\f12\{v;c;v\}=-\f12(v|\o v)\o c=-v\cdot v\ \o c$ since $(v|c)=0$. Hence
$$B_{v,-c}c=c+\{v;c;c\}+Q_vQ_cc=c+v+Q_vc=c+v-v\cdot v\ \o c=\(\f{1-v\cdot v}2,i\f{1+v\cdot v}2,v\)$$
and, more generally, $B_{v,-c}uc=\(u\f{1-v\cdot v}2,iu\f{1+v\cdot v}2,uv\)$ for $u\in\Cl.$ Thus $z_0=u\f{1-v\cdot v}2,\ z_1=iu\f{1+v\cdot v}2$ and $z_{k+1}=u\ v_k$ for $1\le k\le a=n-1$. Since
$$(v_1\ du+u\ dv_1)\yi\ldots\yi(v_a\ du+u\ dv_a)=u^a\ dv_1\yi\ldots\yi dv_a
+u^{a-1}\S_{k=1}^a(-1)^{k-1}v_k\ du\yi dv_1\yi\ldots\yi\widehat{dv_k}\yi\ldots\yi dv_a,$$
it follows that
$$-2i dz_1\yi dz_2\yi\ldots\yi d z_n=((1+v\cdot v)du+2u\ v\cdot dv)\yi(v_1\ du+u\ dv_1)\yi\ldots\yi(v_a\ du+u\ dv_a)$$
$$=u^a(1+v\cdot v)du\yi dv_1\yi\ldots\yi dv_a+2u^a\ v\cdot dv\S_{k=1}^a(-1)^{k-1}v_k\ du\yi dv_1\yi\ldots\yi\widehat{dv_k}\yi\ldots\yi dv_a$$
$$=u^a(1+v\cdot v)du\yi dv_1\yi\ldots\yi dv_a-2u^a\S_{k=1}^a v_k^2\ du\yi dv_1\yi\ldots\yi dv_k\yi\ldots\yi dv_a
=u^a(1-v\cdot v)du\yi dv_1\yi\ldots\yi dv_a.$$
Thus, in agreement with the general formula \er{6i},
$$\f{dz_1\yi\ldots\yi d z_n}{z_0}=i\ u^{a-1}du\yi dv_1\yi\ldots\yi dv_a.$$

For all tube domains, every $h\in\r K_c^0$ satisfies $({\det}_{W^\perp}h)^2=1$ and thus $\det_{W^\perp}h=\pm 1$. More generally, we have $|{\det}_{W^\perp}h|=1$ for all $h\in\r K_c^1$. Thus, as in the maximal rank case, there is always a $\r K_1$-invariant measure
$\lM$ on $\r Z_\el$ which is a multiple of $\lT\yi\o\lT$ if $p-a\el$ is even. Since $\r K_1$ contains $K$ this measure has a polar decomposition under the natural $K$-action.

\begin{theorem} For domains of tube type, the invariant measure $\lM$ on $\r Z_\el$ has the polar decomposition
\be{6j}\I_{\r Z_\el}d\lM(z)\ f(z)=\f{|S_\el|}{2^{d'_\el}}\I_{\lO_c}\f{dt}{N_c(t)^{d'_\el/\el}}\ \ N_c(t)^{ar/2}\I_K dk\ f(k\F t).\ee
\end{theorem}
\begin{proof} Write $\lO_c=G_c/K_c$. Then $N_c(t)^{-d'_\el/\el}\ dt\ ds$ is the $G_c\xx K$-invariant measure on $\lO_c\xx K/K^c.$ Consider the map $\lg:\lO_c\xx K/K^c\to\r Z_\el=\r K_1/\r K_c^1$ defined by $\lg(t,kK^c):=k t$ for all $k\in K$ and $t\in\lO_c.$ Let
$s'=\F t\in\lO_c$ denote the square-root relative to $\lO_c$. Choose $s''\in\lO_{e-c}$ such that $s=s'+s''\in\lO$ satisfies
$\det P_s=1.$ Then $P_s c=P_{s'}c=t$. By \cite[Lemma X.1.9]{H} we have
$$\I_{\r Z_\el}d\lM(z)\ f(z)=\I_{\lO_c\xx K/K^c}d(\lg^*\lM)(t,\d k)\ f(kt)
=\I_{\lO_c}\f{dt}{N_c(t)^{d'_\el/\el}}\ D(t)\I_{K/K^c}d\d k\ f(k t),$$
where the function $D(t)$ on $\lO_c$ is the determinant of the linear transformation
$$T_t(P_s^{-1})\oc T_{c,t}(\lf)\oc T_{c,id}(P_{s'}\xx id_{S_\el})=P_s^{-1}|_{W^\perp}\oc T_{c,t}(\lf)\oc(P_{s'}\oplus id_{W^\perp})$$
on $W^\perp.$ Since ${\det}_{W^\perp}^\Rl P_s=N_c(s')^{2(p-a\el)}=N_c(t)^{p-a\el}$ by Lemma~\ref{6g} and
$\det P_{s'}=N_c(s')^{2d'_\el/\el}=N_c(t)^{d'_\el/\el},$ it follows with Lemma \ref{2c} that
$2^{d'_\el}\ D(t)=N_c(t)^{a\el-p+2a(r-\el)+d'_\el/\el}\ \det F'(t),$ with $a\el-p+2a(r-\el)=a(r+1-\el)-2.$ Hence
$$\I_{\r Z_\el}d\lM(z)\ f(z)=\I_{\lO_c}dt\ N_c(t)^{a(r+1-\el)-2}\ \det F'(t)\I_K dk\ f(k t).$$
Replacing $t$ by $\F t,$ the assertion follows.
\end{proof}

For $\el=1$ \er{6j} yields
$$\I_{\r Z_1}d\lM(z)\ f(z)=|S_1|\I_0^\oo dt\ t^{ra-1}\I_{S_1}ds\ f(ts).$$
For the spin factor $Z$ we have $ar=2(n-1)$ and hence, in agreement with \cite[Lemma 2.1]{MY},
$$\I_{\r Z_1}d\lM(z)\ f(z)=|S_1|\I_0^\oo dt\ \ t^{2n-3}\I_{S_1}ds\ f(ts).$$

Comparing \er{6j} and \er{2e} we obtain
\begin{corollary} In the tube domain case, the invariant measure $\lM$ and the Riemann measure $\lL_\el$ are related by
$d\lL_\el(z)=d\lM(z)\ \lD_\el(z,z)^k,$ where $2k=p-a\el.$
\end{corollary}

For domains of tube type, the invariant measure $\lM$ satisfies $\Lt^*\lM=|N_c(u)|^{a(r+1-\el)-2}\ du\ dv.$ This follows from Proposition~\ref{6h} if $p-a\el$ is even, and in general by using the absolute value of the respective determinants in its proof. It follows that
$$|\det \Lt^{\prime*}_{u,v} \Lt'_{u,v}|\ du\ dv=\Lt^*(\lL_\el)=\ls^*(\lM)\ \lD_\el(\Lt(u,v))^k=N_c(u)^{a(r-\el)-k}\ \lD_\el(\Lt(u,v))^k\ du\ dv.$$
Therefore $|\det \Lt^{\prime*}_{u,v} \Lt'_{u,v}|=N_c(u)^{a(r-\el)-k}\ \lD_\el(\Lt(u,v))^k.$ Since $\Lt'(u,v)
=B_{v,-c}\oc(id_U\oplus D_{u,c}|_V)$ this implies $\det(B_{v,-c}^*B_{v,-c})_{W^\perp}=N_c(u)^{-a(r-\el)-k}\ \lD_\el(\Lt(u,v))^k.$ The latter value can in principle be computed.

\section{Hankel operators}
For a measure $d\t\lr(t)$ on $\lO_c$ invariant under the automorphism group $L_c=\operatorname{Aut}(X_c)$, and the associated radial measure $d\lr$ on
$\r Z_\el$, recall that the Hankel operator~$\Ho g$ on the space $H^2_\lr(\r Z_\el)\equiv H^2_\lr$ with symbol
$g\in H^2_\lr(\r Z_\el),$ is~the operator from $H^2_\lr$ into $L^2(\r Z_\el,d\lr)\equiv L^2_\lr$ defined~by
$$\Ho g f:=(I-P)(\o g f),$$
where $P:L^2_\lr\to H^2_\lr$ is the orthogonal projection. This is a densely defined operator, which~is (extends~to~be) bounded e.g.~whenever $g$ is bounded. For~the (analogously defined) Hankel operators on the unit disc $\Dl$ in~$\Cl$ or the unit ball $\Bl^n$ of~
$\Cl^n,\ n\ge2$, criteria for the membership of $\Ho g$ in the Schatten classes~$\SL^p,\ p>0$, were given in the classical papers by Arazy, Fisher and Peetre \cite{AFPD,AFPB}: For $p\le1$ there are no nonzero $\Ho g$ in $\SL^p$ on~$\Dl$, while for $p>1$, $\Ho g\in\SL^p$ if~and only~if $g\in B^p(\Dl)$, the $p$-th order Besov space on~$\Dl$; while on~$\Bl^n,\ n\ge2$, there are no nonzero $\Ho g$ in $\SL^p$ if $p\le 2n$, while for $p>2n$, again $\Ho g\in\SL^p$ if~and only~if $g\in B^p(\Bl^n)$. One~says that there is a cut-off at $p=1$ or $p=2n$, respectively. The~result remains in force also for $\Dl$ and $\Bl^n$ replaced any bounded strictly-pseudoconvex domain in $\Cl^n$, $n\ge1$, with smooth boundary \cite{Lue}. For~bounded symmetric domains of rank $r>1$, the situation changes drastically: there~are no nonzero compact Hankel operators $\Ho g$ at~all, i.e.~$\Ho g$ is compact only if $g$ is constant \cite[Theorem~D]{BBCZ}.

Returning to the Kepler manifolds, consider the unit ball $\c Z$ of $Z$ and the corresponging `Kepler balls' $\c Z_\el:= Z_\el\ui\c Z,$ and let $d\t\lr(t)=|S_\el| 2^{-d'_\el}N_c(t)^{a(r-\el)+b}\,dt$ (cf.~\er{2e}) be~the measure which makes $d\t\lr$ the Lebesgue surface measure on~$\r Z_\el$ (up~to a constant factor). Except for the top-rank tube domain case $\el=r,\ b=0$, which we exclude in this section, Theorem~\ref{1j} shows that $H^2_\lr$ can be viewed as a space of holomorphic functions on the closure~$Z_\el$, which we will do from now~on. In this situation, it will be shown that the~cut-off phenomenon for Hankel operators is exactly the same as described in the previous paragraph: i.e.~for $\el=1$, the cutoff occurs at $p=2d_1$, while for $\el>2$, $\Ho g$~is never a compact operator if it is nonzero.

\begin{theorem} \label{7a} Let $\c Z_\el$ and $d\lr$ be as above and $p\ge1$.
Then the following are equivalent.
\begin{itemize}
\item[(i)] There exists nonconstant $g\in H^2_\lr$ with $\Ho g\in\SL^p$.
\item[(ii)] There exists a nonzero partition $\m\in\Nl^r_+$ such that  $\Ho g\in\SL^p$ for all $g\in\PL_\m$.
\item[(iii)] $\el=1$ and $p>2d_1$.
\item[(iv)] $\Ho g\in\SL^p$ for any polynomial~$g$.
\end{itemize}
\end{theorem}

\begin{proof} (i)$\implies$(ii) Let $g=\S_\m g_\m$ be the Peter-Weyl expansion of~$g$. For~$k\in K$, consider the rotation operator
$U_k f(z):=f(k^{-1}z)$ for $z\in\c Z_\el.$ Clearly $U_k$ is unitary on $H^2_\lr$ as well as on~$L^2_\lr$, and
$$U_k\Ho g U_k^*=\Ho{U_k g}.$$
Thus also $\Ho{U_k g}\in\SL^p$ for all $k\in K$. Furthermore, the~action $k\mapsto U_k$ is continuous in the strong operator topology, i.e.~$k\mapsto U_k f$ is norm continuous for each $f\in L^2_\lr$. By \cite[Lemma on p.~997]{AFPD} this implies that the map $k\mapsto \Ho{U_k g}$ is~even continuous from $K$ into~$\SL^p$. Consequently, denoting by $\lc_\m$ the character of the representation $\lp_\m$ of $K$ on~$\PL_\m$, the~Bochner integral
$$\I_K dk\ \lc_\m(k)\ U_k\Ho g U^*_k=\Ho{g_\m}$$
also belongs to~$\SL^p$, for~each~$\m$. As~$g$ is nonconstant, there exists nonzero $\m\in\Nl^r_+$ for which $g_\m\ne 0$. Then, again,
$\Ho{U_k g_\m}=U_k\Ho{g_\m}U^*_k\in\SL^p$ for all $k\in K$; since, by~irreducibility, the~translates $U_k g_\m=\pi_\m(k)g_\m$ span all of~$\PL_\m$, (ii)~follows.

(ii)$\implies$(iii) We~first show that $\Ho g$, with $g\in\PL_\m$ nonconstant, can~never be compact if $\el>1$. To~this end, assume that $\el\ge2$, and take the polydisc $D:=\Dl e_1+\dots+\Dl e_\el\subset\c Z_\el$. Denote temporarily by $R_D$ the restriction map
$f\mapsto f|_D$, and by $\Ho h^D:f\mapsto(I-P_D)(\o hf)$ the Hankel operator with symbol $\o h$ on~$D$, where $P_D$ is the Bergman
projection on~$D$. Then
$$(I-P_D)R_D\Ho g f=(I-P_D)R_D(gf-P(gf))=\Ho{R_D g}^D R_D f-(I-P_D)R_D P(gf)=\Ho{R_D g}^D R_D f,$$
because, since a~restriction to $D$ of a holomorphic function is again holomorphic, $(I-P^D) R_D P=0$. Thus
$(I-P_D)R_D\Ho g=\Ho{R_D g}^D R_D$, hence if $\Ho g$ is compact, then~so is $\Ho{R_D g}^D R_D$. Since $\c Z\ic D\xx\Cl^{d-\el}$
(see~\cite[Theorem~2.5~(ii)]{Timo}), the~restriction operator $R_D$ is surjective (a~holomorphic function in the Bergman space of $D$ extends to a function in $H^2_\lr$ by making it independent of the superfluous $d-\el$ variables; here we are using the fact that
$d\lr$ was taken to be the Lebesgue measure). Thus $\Ho{R_D g}^D$ is compact. By~the well-known result recalled in the previous paragraph (which for the polydisc in fact goes back to \cite[Proposition~4.1]{Timo}), $R_D g$ must be constant. This contradicts the fact that $R^D\PL_\m$ contains, for~instance, the~nonconstant function $z_1^{m_1}\dots z_\el^{m_\el}$ on~$D$.

The~hypothesis (ii) therefore implies that necessarily $\el=1$. In~that case, $\c Z_1$ is just the Euclidean unit ball
$\{z\in Z_1:\|z\|<1\}$ of~$Z_1$, in~particular, it~is a strictly pseudoconvex domain with smooth boundary
(albeit with a singularity at the origin). Now~a general theory of Hankel operators on smoothly bounded strictly pseudoconvex domains in $\Cl^n$ with symbols smooth on the closure of the domain --- even not necessarily anti-holomorphic as in this paper
--- was~developed by one of the authors and G.~Zhang in \cite{EZ}, showing in particular that $\Ho g$, with $g$ smooth on the closure, belongs to the weak Schatten ideal $\SL^{2n,\oo}$, and furthermore, $(\Ho g^*\Ho g)^n$ belongs to the Dixmier class and its Dixmier trace is given, in~the particular case of the spherical boundary we have here, just by an integral over the `unit sphere'
$\{z\in Z_1:\|z\|=1\}$ of $|\o\dl_b\o g|^2$ (up~to~a constant factor), where $\o\dl_b$ is the boundary Cauchy-Riemann operator.
(See \cite[Theorem~11 and formula (24)]{EZ} for the details.) In~particular, if~$g$ is holomorphic and nonconstant, then $\o\dl_b\o g$ does not vanish identically, hence the Dixmier trace of $(\Ho g^*\Ho g)^n$ is positive, implying in particular that $\Ho g$ cannot belong to any $\SL^p$ with $p\le 2n$. Finally, the~proofs in \cite{EZ} relied on the theory, due~to Boutet de Monvel and Guillemin, of~generalized Toeplitz operators (with pseudo-differential symbols), which remains in force also if domains in $\Cl^n$ are replaced by domains in complex varieties which can have singularities in the interior but not on the boundary, cf.~\oldS 2.i~in~\cite{BDMi}; in~particular, this theory --- and, hence, also~all the results in~\cite{EZ} --- remain in force also in our situation when the domain is the unit ball in the complex variety $Z_1$ with the sole singularity at the origin. Altogether, it~thus follows that if $\el=1$ and $g\in\PL_\m$ is nonconstant, then $\Ho g\in\SL^{2n,\oo}\ic\SL^p$ for all $p>2n$, $n:=\dim_\Cl\r Z_1$, but $\Ho g\notin\SL^p$ for any $p\le2n$. Thus (iii) holds. (The~beginning of the proof of this implication was inspired by the proof of \cite[Proposition~1]{CuSa}.)

(iii)$\implies$(iv) We~have seen in the proof of (ii)$\implies$(iii) that for $\el=1$ and $p>2d_1$, $\Ho g\in\SL^p$ even for any $g$ smooth on the closure of~$\c Z_1$. In~particular, since polynomials are smooth on all of~$Z$, this will certainly be the case for any polynomial~$g$.

Since the implication (iv)$\implies$(i) is trivial, this completes the proof of the theorem.
\end{proof}

\end{document}